\documentclass[a4paper,reqno,11pt]{amsart}

\usepackage{amsmath, amsfonts, amssymb, amsthm, amscd}
\usepackage{graphicx}
\usepackage{psfrag}
\usepackage{perpage}
\usepackage{url}
\usepackage{color}
\usepackage{mathrsfs}
\usepackage{mathabx}
\usepackage{scalerel}
\usepackage{stmaryrd}
\usepackage{framed} 

\usepackage{tikz}\usepackage{caption,subcaption}
\usetikzlibrary{arrows,decorations.pathmorphing,backgrounds,positioning,fit,petri}
\usepackage{subdepth}

\usepackage{dsfont} 

\usepackage[T1]{fontenc}
\usepackage{microtype}
\usepackage{lmodern}

\usepackage[a4paper,scale={0.72,0.74},marginratio={1:1},footskip=7mm,headsep=10mm]{geometry}

\usepackage{hyperref}

\makeatletter
\def\@secnumfont{\bfseries\scshape}

\def\section{\@startsection{section}{1}
  \z@{.9\linespacing\@plus\linespacing}{.5\linespacing}%
  {\normalfont\large\bfseries\scshape\centering}}

\def\subsection{\@startsection{subsection}{2}%
  \z@{.5\linespacing\@plus.7\linespacing}{-.5em}%
  {\normalfont\bfseries\scshape}}

\def\subsubsection{\@startsection{subsubsection}{3}%
  \z@{.5\linespacing\@plus.7\linespacing}{-.5em}%
  {\normalfont\scshape}}

\def\specialsection{\@startsection{section}{1}%
  \z@{\linespacing\@plus\linespacing}{.5\linespacing}%
  {\normalfont\centering\large\bfseries\scshape}}
\makeatother

\makeatletter

\renewenvironment{proof}[1][\proofname]{\par
\pushQED{\qed}%
\normalfont \topsep4\p@\@plus4\p@\relax
\trivlist
\item[\hskip\labelsep
\bfseries
#1\@addpunct{.}]\ignorespaces
}{%
\popQED\endtrivlist\@endpefalse
}
\makeatother

\makeatletter
\newcommand \Dotfill {\leavevmode \leaders \hb@xt@ 6pt{\hss .\hss }\hfill \kern \z@}
\makeatother

\makeatletter
\def\@tocline#1#2#3#4#5#6#7{\relax
  \ifnum #1>\c@tocdepth 
  \else
    \par \addpenalty\@secpenalty\addvspace{#2}%
    \begingroup \hyphenpenalty\@M
    \@ifempty{#4}{%
      \@tempdima\csname r@tocindent\number#1\endcsname\relax
    }{%
      \@tempdima#4\relax
    }%
    \parindent\z@ \leftskip#3\relax \advance\leftskip\@tempdima\relax
    \rightskip\@pnumwidth plus4em \parfillskip-\@pnumwidth
    #5\leavevmode\hskip-\@tempdima
      \ifcase #1
       \or\or \hskip 1.65em \or \hskip 3.3em \else \hskip 4.95em \fi%
      #6\nobreak\relax
    \Dotfill
    \hbox to\@pnumwidth{\@tocpagenum{#7}}\par
    \nobreak
    \endgroup
  \fi}
\makeatother

\makeatletter
\def\l@section{\@tocline{1}{0pt}{1pc}{}{\scshape}}
\renewcommand{\tocsection}[3]{%
\indentlabel{\@ifnotempty{#2}{\ignorespaces#1 #2.\hskip 0.7em}}#3}
\def\l@subsection{\@tocline{2}{0pt}{1pc}{5pc}{}}

\def\l@subsubsection{\@tocline{3}{0pt}{1pc}{7pc}{}}

\makeatother

\numberwithin{equation}{section}

\definecolor{shadecolor}{gray}{.94}

\makeatletter
\newenvironment{myshade}{%
  \topsep4\p@\@plus4\p@\relax%
  \MakeFramed{\advance\hsize-\width \FrameRestore}}%
 {\par\unskip\endMakeFramed}%
\makeatother

\newtheoremstyle{mytheorem}{0}{0}%
     {\itshape}
     {}
     {\bfseries}
     {. }
     {0.3ex}
     {\thmname{{\bfseries #1}}\thmnumber{ {\bfseries #2}}\thmnote{ (#3)}}  

\theoremstyle{mytheorem}

\newtheorem{theo}{Theorem}[section]
\newenvironment{theorem}{\begin{myshade}\begin{theo}}{\end{theo}\end{myshade}}

\newtheorem{prop}[theo]{Proposition}
\newenvironment{proposition}{\begin{myshade}\begin{prop}}{\end{prop}\end{myshade}}

\newtheorem{lem}[theo]{Lemma}
\newenvironment{lemma}{\begin{myshade}\begin{lem}}{\end{lem}\end{myshade}}

\newtheorem{defin}[theo]{Definition}
\newenvironment{definition}{\begin{myshade}\begin{defin}}{\end{defin}\end{myshade}}

\newtheorem{cor}[theo]{Corollary}

\newtheorem{assump}[theo]{Assumption}
\newenvironment{assumption}{\begin{myshade}\begin{assump}}{\end{assump}\end{myshade}}

\newtheoremstyle{myremark}{.7\linespacing\@plus.3\linespacing}{.7\linespacing\@plus.3\linespacing}%
     {\itshape}
     {}
     {\bfseries}
     {. }
     {0.3ex}
     {\thmname{{\bfseries #1}}\thmnumber{ {\bfseries #2}}\thmnote{ (#3)}}  

\theoremstyle{myremark}

\newtheorem{remark}[theo]{Remark}

\newtheoremstyle{mydefinition}{.7\linespacing\@plus.3\linespacing}{.7\linespacing\@plus.3\linespacing}%
     {\rmfamily}
     {}
     {\bfseries}
     {. }
     {0.3ex}
     {\thmname{{\bfseries #1}}\thmnumber{ {\bfseries #2}}\thmnote{ (#3)}}  

\theoremstyle{mydefinition}

\newtheorem{example}[theo]{Example}

\newcommand{\bbE}{{\ensuremath{\mathbb E}} }

\newcommand{\bbP}{{\ensuremath{\mathbb P}} }

\newcommand{\bbT}{{\ensuremath{\mathbb T}} }

\newcommand{\cA}{{\ensuremath{\mathcal A}} }

\newcommand{\cE}{{\ensuremath{\mathcal E}} }
\newcommand{\cF}{{\ensuremath{\mathcal F}} }

\newcommand{\cV}{{\ensuremath{\mathcal V}} }
\newcommand{\cW}{{\ensuremath{\mathcal W}} }

\newcommand{\cZ}{{\ensuremath{\mathcal Z}} }

\DeclareMathSymbol{\leqslant}{\mathalpha}{AMSa}{"36} 
\DeclareMathSymbol{\geqslant}{\mathalpha}{AMSa}{"3E} 
\DeclareMathSymbol{\eset}{\mathalpha}{AMSb}{"3F}     

\newcommand{\sumtwo}[2]{\sum_{\substack{#1 \\ #2}}} 

\newcommand{\be}{\begin{equation}}
\newcommand{\ee}{\end{equation}}

\newcommand{\R}{\mathbb{R}}

\newcommand{\Z}{\mathbb{Z}}
\newcommand{\N}{\mathbb{N}}

\newcommand{\PEfont}{\mathrm}

\newcommand{\p}{\ensuremath{\PEfont P}}

\newcommand{\E}{\ensuremath{\PEfont  E}}
\renewcommand{\P}{\p}

\DeclareMathOperator{\bbvar}{\ensuremath{\mathbb{V}ar}}
\DeclareMathOperator{\bbcov}{\ensuremath{\mathbb{C}ov}}

\newcommand{\ind}{\mathds{1}}

\renewcommand{\epsilon}{\varepsilon}
\renewcommand{\theta}{\vartheta}
\renewcommand{\rho}{\varrho}

\newenvironment{myenumerate}{%
\renewcommand{\theenumi}{\arabic{enumi}}%
\renewcommand{\labelenumi}{{\rm(\theenumi)}}%
\begin{list}{\labelenumi}
	{%
	\setlength{\itemsep}{0.4em}%
	\setlength{\topsep}{0.5em}%
	\setlength\leftmargin{2.45em}%
	\setlength\labelwidth{2.05em}%
	\setlength{\labelsep}{0.4em}%
	\usecounter{enumi}%
	}%
	}%
{\end{list}
}

\newenvironment{aenumerate}{%
\renewcommand{\theenumi}{\alph{enumi}}%
\renewcommand{\labelenumi}{{\rm(\theenumi)}}%
\begin{list}{\labelenumi}
	{%
	\setlength{\itemsep}{0.4em}%
	\setlength{\topsep}{0.5em}%
	\setlength\leftmargin{2.45em}%
	\setlength\labelwidth{2.05em}%
	\setlength{\labelsep}{0.4em}%
	\usecounter{enumi}%
	}%
	}%
{\end{list}
}

{\end{list}
}

{\end{myenumerate}}

\newenvironment{myitemize}{%
\begin{list}{$\bullet$}%
 	{%
	\setlength{\itemsep}{0.4em}%
	\setlength{\topsep}{0.5em}%
	\setlength\leftmargin{2.65em}%
	\setlength\labelwidth{2.65em}%
	\setlength{\labelsep}{0.4em}%
	}%
	}%
{\end{list}}

\renewenvironment{itemize}{
\begin{myitemize}}%
{\end{myitemize}}

\MakePerPage[2]{footnote} 

\date{\today}

\newcommand\dd{\mathrm{d}}

\newcommand\bell{\boldsymbol{\ell}}

\newcommand\rme{\mathrm{e}}

\newcommand\Inf{\mathrm{Inf}}

\usepackage{scalerel,stackengine}
\newcommand\pig[1]{\scalerel*[5pt]{\big#1}{%
  \ensurestackMath{\addstackgap[1.5pt]{\big#1}}}}
\newcommand\pigl[1]{\mathopen{\pig{#1}}}
\newcommand\pigr[1]{\mathclose{\pig{#1}}}

\definecolor{magenta}{rgb}{1.0, 0.0, 1.0}

\begin{document}

\author[F. Caravenna]{Francesco Caravenna}
\address{Dipartimento di Matematica e Applicazioni\\
 Universit\`a degli Studi di Milano-Bicocca\\
 via Cozzi 55, 20125 Milano, Italy}
\email{francesco.caravenna@unimib.it}

\author[A. Donadini]{Anna Donadini}
\address{Dipartimento di Matematica e Applicazioni\\
 Universit\`a degli Studi di Milano-Bicocca\\
 via Cozzi 55, 20125 Milano, Italy}
\email{a.donadini@campus.unimib.it}

\title[Enhanced noise sensitivity and SHF]{Enhanced noise sensitivity, 2D directed polymers\\
and Stochastic Heat Flow}


\begin{abstract}
We investigate noise sensitivity beyond the classical setting of
binary random variables,
extending the celebrated result by Benjamini, Kalai, and Schramm
to a wide class of functions of general random variables. 
Our approach yields improved bounds with optimal rates.
We also consider an enhanced form of noise
sensitivity which yields asymptotic independence, rather than mere decorrelation.

We apply these result to establish enhanced noise sensitivity for the partition functions of 2D 
directed polymers, in the critical regime where they converge to the 
critical 2D Stochastic Heat Flow.
As a consequence, we prove that the Stochastic Heat Flow is
independent of the white noise arising from the disorder.
\end{abstract}

\maketitle

\section{Introduction}

Given a function $f(\omega)$ that depends
on independent random variables~$\omega = (\omega_1, \omega_2, \ldots)$, 
the concept of \emph{noise sensitivity} describes the intriguing phenomenon where
a small perturbation of the variables~$\omega_i$ completely alters the function's output.
This phenomenon is particularly relevant in the study of 
Boolean functions \cite{O}, where small changes in input bits can make the output unpredictable, and in 
statistical physics \cite{GS}, where it describes the sensitivity of systems to random noise, especially 
in the proximity of a phase transition.

\smallskip

The literature on noise sensitivity has largely focused on 
binary variables \( \omega_i \) taking two values,  say \( x_+ \) and \( x_- \):
\begin{equation} \label{eq:binary0}
	\bbP(\omega_i = x_+) = p \quad \text{and} \quad
	\bbP(\omega_i = x_-) = 1-p \quad \text{for some \( p \in (0, 1) \)} \,.
\end{equation}
Most results also concern Boolean functions \( f(\omega) \in \{0, 1\} \)
(see references below).
In this paper, we extend the classical BKS criterion for noise sensitivity 
\cite{BKS} to a general setting, allowing for a wide
class of \emph{non-Boolean functions \( f(\omega) \in \R \) 
of general random variables~\( \omega_i \)} (see Theorems~\ref{th:ns-gen}-\ref{th:BKS-gen}).
When specialised to the binary setting, our results yield
\emph{asymptotically optimal rates} (see Theorems~\ref{th:BKS-optimal}-\ref{th:sharp}).

\smallskip

The concept of noise sensitivity is usually formulated
in terms of vanishing correlations (see \eqref{eq:ns}) which
--- for Boolean function $f(\omega) \in \{0,1\}$ --- corresponds to asymptotic independence.
For general real functions \( f(\omega) \in \R \), however, it is natural to consider
an \emph{enhanced form of noise sensitivity} (see \eqref{eq:ns+}), where $f(\omega)$
is composed with suitable test functions to ensure 
asymptotic independence, rather than mere decorrelation.
We show that \emph{the BKS criterion implies enhanced 
noise sensitivity}, at least when the random variables $\omega_i$
take finitely many values (see Theorem~\ref{th:BKS+}).

\smallskip

Our results on noise sensitivity are presented in Section~\ref{sec:ns}.
In the following Section~\ref{sec:appl}, we apply them to a model from statistical mechanics: the two-dimensional directed polymer in random environment.
We first establish an instance of \emph{enhanced noise sensitivity for the partition functions} (see Theorem~\ref{thm:noise_sensitivity_polymer}), in the critical regime where they converge to the Critical 2D Stochastic Heat Flow (SHF) \cite{CSZ23}.
Next, we prove the \emph{independence of the SHF from the white noise originating from the environment} (see Theorem~\ref{th:SHF}), which indicates that the SHF is not the solution of a stochastic PDE driven by white noise.

\smallskip

In the related setting of the two-dimensional Stochastic Heat Equation with spatially regularised white noise, the convergence of the solution to the SHF was recently established in~\cite{T24}.
An analogue of our Theorem~\ref{th:SHF} in this context --- proving independence of the SHF from the white noise --- was obtained independently and simultaneously in~\cite{GT}, as a corollary of the main result of the paper, which shows that the SHF is a so-called black noise.
We refer the reader to the beginning of Section~\ref{sec:appl} for a more detailed discussion of the SHF and further references.

\smallskip

The following Sections~\ref{sec:setting}-\ref{sec:optimality}
are devoted to the proof of our results, while some technical points
are deferred to the appendices.

\subsection*{Acknowledgements}

We are very much indebted to Christophe Garban, who introduced
one of us (F.C.) to noise sensitivity and suggested to investigate it for directed polymers.

We are also grateful to Malo Hillairet, Giovanni Peccati and Hugo Vanneuville for enlightening discussions
and for pointing out relevant references.

We thank Daniel Ahlberg, Malo Hillairet and Ekaterina Toropova for generously 
sharing a preliminary version
of their manuscript \cite{AHT}, in which they independently prove a generalisation of the BKS criterion,
using different techniques.

We gratefully acknowledge the support of INdAM/GNAMPA.

\section{Main results for noise sensitivity}
\label{sec:ns}

In this section, we present our main results concerning noise sensitivity.
Let us first recall classical results in the setting of binary random variables $\omega_i$'s,
see \eqref{eq:binary0}.

\subsection{Binary setting}

Given independent random variables~$\omega = (\omega_1, \omega_2, \ldots)$,
we denote by $\omega^\epsilon = (\omega^\epsilon_1, \omega^\epsilon_2, \ldots)$ the 
configuration obtained 
by independently resampling each $\omega_i$ with probability~$\epsilon$
(see Definition~\ref{def:rand}).
We focus in this subsection on binary $\omega_i$'s, see \eqref{eq:binary0}.

A sequence of Boolean functions $f_N(\omega) \in \{0,1\}$ is called \emph{noise sensitive} if
\begin{equation} \label{eq:ns}
	\forall \epsilon > 0 \colon \qquad
	\lim_{N\to\infty} \, \bbcov\big[ f_N(\omega^\epsilon), f_N(\omega) \big] = 0 \,,
\end{equation}
which implies the asymptotic independence of $f_N(\omega^\epsilon)$ and $f_N(\omega)$.
This notion was introduced by I.~Benjamini, G.~Kalai and O.~Schramm  
in their seminal paper \cite{BKS}.

It was shown in the same paper  that noise sensitivity is closely related to the probability that 
flipping a single variable $\omega_k$ changes the output of the function $f(\omega)$. 
To formalize this, the \emph{influence} of $\omega_k$ on $f(\omega)$ is defined as
\begin{equation} \label{eq:GS-influence}
	I_k(f) := \bbP \big( f(\omega^k_+) \ne f(\omega^k_-) \big) \,,
\end{equation}
where $\omega^k_\pm$ denotes the configuration $\omega = (\omega_1, \omega_2, \ldots)$ 
with $\omega_k$ fixed to $x_\pm$.
The main result of \cite{BKS} establishes a sufficient
``BKS criterion'' for noise sensitivity, based 
on the \emph{sum of squared influences}:
for any sequence of Boolean functions 
$f_N(\omega) \in \{0,1\}$, we have
\begin{equation}\label{eq:BKS-original}
	\sum_{k} I_k(f_N)^2 \xrightarrow[\, N\to\infty \, ]{} 0
	\qquad \Longrightarrow \qquad
	\text{$(f_N)_{N\in\N}$ is noise sensitive} \,.
\end{equation}
This result was originally proved in the symmetric case \( p = \frac{1}{2} \) in \cite{BKS},
but remarked to hold  for any  \( p \in (0, 1) \); a proof was given in \cite{ABGR14}.

Intuitively, the sum of squared influences measures how much the function \(f_N\) depends on 
individual variables~$\omega_k$. 
If this sum vanishes as \(N \to \infty\), it indicates that no single variable has 
a significant impact on the output, which must depend jointly on many variables,
making the function highly sensitive to noise.

A quantitative refinement of the BKS criterion \eqref{eq:BKS-original} was provided
by N.~Keller and G.~Kindler \cite[Theorem~7]{KK}:
for any \( \epsilon \in (0,1) \) there is an exponent
$\gamma_\epsilon = \gamma_{\epsilon,p} > 0$ such that,
for all Boolean functions \( f(\omega) \in \{0, 1\} \), one can bound
\begin{equation}\label{eq:KK-original}
	\bbcov\big[ f(\omega^\epsilon), f(\omega) \big] \le
	20 \, \cW[f]^{\gamma_{\epsilon}} \qquad \text{with} \qquad
	\cW[f] := 4p(1-p) \, \sum_k I_k[f]^2 \,.
\end{equation}
This was improved in \cite{EG22}, where it was shown that
--- always for Boolean functions $f(\omega) \in \{0,1\}$ ---
the RHS can be multiplied by $\bbvar[f]$.

The estimate \eqref{eq:KK-original} 
shows that the covariance between \( f(\omega^\epsilon) \) 
and \( f(\omega) \) is controlled by a power of the sum of squared influences \( \cW[f] \)
and, clearly, it implies the BKS criterion~\eqref{eq:BKS-original}.
The exponent $\gamma_\epsilon$ is asymptotically linear as $\epsilon \downarrow 0$,
namely $\gamma_{\epsilon} \sim \alpha \, \epsilon$
for a suitable $\alpha = \alpha(p) > 0$
(the asymptotic notation $\sim$ means that the ratio of the two sides converges to~$1$).

\smallskip

It was already remarked in \cite{KK,ABGR14} that the BKS criterion
\eqref{eq:BKS-original} holds beyond Boolean functions.
For instance, the proof of the estimate \eqref{eq:KK-original}
in \cite{KK} goes through for any function $f(\omega) \in [0,1]$,
provided one extends the definition \eqref{eq:GS-influence} of influence 
as follows:
\begin{equation} \label{eq:inf-quasi-gen}
	I_k[f] = \bbE[|f(\omega_+^k) - f(\omega_-^k)|] \,.
\end{equation}

More recently, 
building on ideas from \cite{R06,FS07},
it was observed by R.\ van Handel in a series of lectures
(reported in \cite{Ros20}) that a variation of \eqref{eq:KK-original}  holds,
in the symmetric case $p=\frac{1}{2}$, 
\emph{for general functions $f(\omega) \in L^2$}:
for any $\epsilon \in (0,1)$,
setting $\theta_\epsilon = \tfrac{\epsilon}{2-\epsilon}$, one has
\begin{equation}\label{eq:vH}
	\bbcov\big[ f(\omega^\epsilon), f(\omega) \big] \le
	\bbvar[f]^{1-\theta_{\epsilon}} \, \cW[f]^{\theta_{\epsilon}} 
\end{equation}
(the exponent $\theta_\epsilon$ improves on $\gamma_\epsilon$ from
\eqref{eq:KK-original}, 
since $\gamma_\epsilon \sim 0.234 \, \epsilon$ as $\epsilon \downarrow 0$,
while $\theta_\epsilon > 0.5\, \epsilon$).

The bound \eqref{eq:vH} can be extended to general binary variables $\omega_i$ with $p \in (0,1)$,
for a suitable $\theta_{\epsilon,p} > 0$, as shown in \cite{IZ}.
See also the forthcoming work \cite{AHT}, which studies noise sensitivity 
with applications to Last Passage Percolation.

\medskip

All the above results are formulated for binary variables $\omega_i$.
Our goal is to go \emph{beyond the 
binary setting}, extending the BKS criterion \eqref{eq:BKS-original} to general
random variables $\omega_i$'s and to a large class of functions $f(\omega) \in L^2$.
We establish a quantitative bound like \eqref{eq:vH},
with a suitable extension of the ``$L^1$ notion'' \eqref{eq:inf-quasi-gen} of influence.
We also investigate an enhanced notion of noise sensitivity, which yields asymptotic independence
beyond Boolean functions.

\smallskip

In the next Subsection~\ref{sec:set} we formulate our setting and assumptions.
In the following Subsection~\ref{sec:infl} we extend the notion of influence.
Subsections~\ref{sec:general}--\ref{sec:optimal} are devoted to the
presentation of our main results, followed by some concluding remarks
in Subsection~\ref{sec:concluding}.

\begin{remark}
Even though we focus on the setting of a finite or countable family of random variables 
$(\omega_i)_{i\in\bbT}$, we point out that noise sensitivity can also be studied in a continuum setting.
We refer to \cite{LPY23,BPY24} for criteria ensuring noise sensitivity for 
Poisson point processes, with applications
to problems involving continuum percolation.
\end{remark}

\subsection{Our setting}
\label{sec:set}

We consider independent random variables $\omega = (\omega_i)_{i \in \mathbb{T}}$,
labeled by a finite or countable set~$\mathbb{T}$,
defined on a probability space\footnote{We could work on the canonical space $\Omega = \bigtimes_{i\in \bbT} E_i$,
taking $\omega_i$ as the coordinate projections, but it is sometimes convenient to 
allow for extra randomness, e.g.\ to define $\omega^\epsilon$.} $(\Omega, \mathcal{A}, \mathbb{P})$
and taking values in measurable spaces $(E_i, \cE_i)$:
\begin{equation}\label{eq:Ei}
	\omega_i \colon \Omega \to E_i \qquad \text{with law} \qquad
	\mu_i(\cdot) = \bbP(\omega_i \in \cdot) \,.
\end{equation}
In many examples the $\omega_i$'s take values in the same space $(E_i, \cE_i) = (E,\cE)$
and they are i.i.d.\ (independent and identically distributed), but we do not require it.

We are interested in functions $f(\omega) \in L^2$
(where $f: \bigtimes_{i\in\bbT} E_i \to \R$ is measurable). 
We anticipate that, \emph{when the $\omega_i$'s are i.i.d.\
and take finitely many values, we do not need to impose any further constraint on~$f(\omega)$}.
For general $\omega_i$'s, we need to require a suitable hypecontractivity condition,
that we now describe and illustrate by examples.

\smallskip

Given an index $i \in \bbT$, we can look at \emph{$f(\omega)$ as a function of~$\omega_i$},
denoted by $\omega_i \mapsto f(\omega)$, keeping all other variables $(\omega_j)_{j\ne i}$ fixed.
We will single out, for each $i\in \bbT$, a vector space $\cV_i$ of functions of~$\omega_i$
which a.s.\ contains $\omega_i \mapsto f(\omega)$, that is
\begin{equation}\label{eq:belongs}
	\forall i \in \bbT\,, \ \text{ for a.e.\ $(\omega_j)_{j\ne i}\,$}\colon \qquad
	\omega_i \mapsto f(\omega) \
	\text{ belongs to } \ \cV_i \subseteq L^2(E_i, \mu_i)\,.
\end{equation}
Of course, the simplest choice would be to just take $\cV_i \equiv L^2(E_i, \mu_i)$, so that
\eqref{eq:belongs} holds for any $f(\omega) \in L^2$ (by Fubini's theorem). The reason for 
choosing a smaller $\cV_i \subsetneq L^2(E_i, \mu_i)$ is to ensure 
a \emph{hypercontractivity bound}: we require that
centred functions $g\in\cV_i$ have $L^q$ norm controlled by the $L^2$ norm,
for some $q > 2$. This is our main assumption.

\begin{assumption}
\label{hyp:gen2}
The independent random variables $(\omega_i)_{i\in\bbT}$ and the function
$f(\omega) \in L^2$ satisfy the following conditions.
\begin{aenumerate}
\item\label{it:a} For any $i\in\bbT$, there is a closed and separable
vector space $\cV_i \subseteq L^2(E_i, \mu_i)$
containing the constants, i.e.\ $1 \in \cV_i$, such that
$\omega_i \mapsto f(\omega) \in \cV_i$ a.s.\ (that is, \eqref{eq:belongs} holds).

\item\label{it:b} There exist an exponent $q \in (2,\infty)$ and a constant $M_q < \infty$ such that,
for any $i\in\bbT$,
\begin{equation} \label{eq:Mq}
	\|g(\omega_i) \|_q \le M_q \, \|g(\omega_i) \|_2
	\qquad 
	\forall g \in \cV_i 
	\quad \text{with} \quad \bbE[g(\omega_i)] = 0 \,,
\end{equation}
where $\|\cdot\|_q = \mathbb{E} [| \cdot |^{q}]^{1/q}$ denotes the usual $L^q$ norm.
\end{aenumerate}
\end{assumption}

This assumption is quite general and covers a variety of settings.
We present here two relevant examples,
deferring a more detailed discussion to Appendix~\ref{app:hyper},
where we connect Assumption~\ref{hyp:gen2} to the notion of
\emph{ensembles} from \cite[Definition~3.1]{MOO}.
We point out that a similar assumption also appears in \cite[Proposition~2.1]{BR}.

\smallskip

The first example concerns binary or, more generally, finite valued random variables $\omega_i$.
In this setting, we allow for
\emph{arbitrary functions $f(\omega) \in L^2$}.

\begin{example}[Finite support]\label{ex:fin-supp}
Let the random variables $(\omega_i)_{i\in\bbT}$ be i.i.d.\
and take \emph{finitely many values}, i.e.\ their law $\mu_i = \mu$ has \emph{finite support}
($\mu(I) = 1$ with $|I| < \infty$).

Then \emph{Assumption~\ref{hyp:gen2} is satisfied by every function $f(\omega) \in L^2$},
for any exponent $q > 2$ (just choose $\cV_i = L^2(E,\mu) = L^q(E,\mu)$)
and for a suitable $M_q < \infty$ depending only on~$\mu$.
We refer to Lemma~\ref{th:finite-gen} and Remark~\ref{rem:appl-fin} for an elementary proof.
\end{example}

The second example concerns general real valued random variables $\omega_i$'s
satisfying a moment condition.
In this setting, we can allow for any \emph{polynomial chaos} $f(\omega) \in L^2$.
This will be relevant for the application to directed polymers in Section~\ref{sec:appl}.

\begin{example}[Polynomial chaos]\label{ex:poly-chaos}
Let the $(\omega_i)_{i\in\bbT}$ be independent real random variables
with zero mean, unit variance and a uniformly bounded moment of some order $q > 2$:
\begin{equation}  \label{eq:omega}
	\mathbb{E} [\omega_i] = 0 \,, \qquad 
	\mathbb{E} [\omega_i^2] = 1 \,, \qquad
	M_q := \sup_{i\in\bbT} \| \omega_i \|_q < \infty \,.
\end{equation}

Then Assumption~\ref{hyp:gen2} is satisfied by all $f(\omega) \in L^2$
that are \emph{polynomial chaos}, i.e.\ multi-linear polynomials (or power series),
where each variable $\omega_i$ appears with power~$0$ or~$1$:
\begin{equation} \label{eq:fpoly}
	f(\omega) = \bbE[f] + \sum_{d=1}^\infty f^{(d)}(\omega) \qquad \text{with} \quad\
	f^{(d)}(\omega) = \!\!\!\!\!\! \sumtwo{\{i_1, \ldots, i_d\} \subseteq \bbT\\
	i_j \ne i_k \, \forall j \ne k}
	\!\! \hat{f}(i_1,\ldots, i_d) \, \omega_{i_1} \cdots\, \omega_{i_d} 
\end{equation}
(for coefficients $\hat{f}(\cdot)$ such that $\sum_{\{i_1, \ldots, i_d\}} 
\hat{f}(i_1,\ldots, i_d)^2 < \infty$, so the sum converges in~$L^2$).

Since $\omega_i \mapsto f(\omega)$ is a \emph{linear} map,
we see that condition \eqref{eq:belongs} holds with the choice
$\cV_i = \big\{ g(x) = \alpha + \beta \, x \colon \ \alpha, \beta \in \R  \big\}$
and the bound \eqref{eq:Mq} holds by assumption \eqref{eq:omega}.
\end{example}

\begin{remark}[Beyond polynomial chaos] \label{rem:high}
Assumption~\ref{hyp:gen2} is also satisfied by 
any $f(\omega) \in L^2$ which,
as a function of each~$\omega_i$ for fixed $(\omega_j)_{j\ne i}$, 
is a polynomial of
\emph{uniformly bounded degree}, provided the $\omega_i$'s have
enough finite moments; see Remark~\ref{rem:appl-fin2} for details.
\end{remark}

\begin{remark}[Binary setting]\label{rem:binarychaos}
For binary $\omega_i$'s, see \eqref{eq:binary0}, every function $f(\omega) \in L^2$
is a polynomial chaos,
by the Fourier-Walsh expansion (see Lemma~\ref{lem:completeness}). Thus
Examples~\ref{ex:fin-supp} and~\ref{ex:poly-chaos} both extend the classical setting
of binary $\omega_i$'s, allowing for arbitrary $f(\omega) \in L^2$.
\end{remark}

\subsection{General influences}
\label{sec:infl}

We extend the classical notion of influence \eqref{eq:GS-influence} to 
general functions $f(\omega)$ of general random variables $(\omega_i)_{i\in\bbT}$.

For $k\in\bbT$, we introduce the ``probabilistic
 gradient'' of $f$ with respect to~$\omega_k$:\footnote{This
extends a definition of Talagrand~\cite{Tal} for binary $\omega_i$'s
(see Remark~\ref{rem:Tal}). The same notion, denoted $\Delta_k f$,
also appears in \cite{BR}.}
\begin{equation}\label{eq:deltak}
	\delta_k f := f - \bbE_k[f]  \qquad
	\text{where we set} \quad
	\bbE_k[\cdot] := \bbE[\cdot | \sigma((\omega_j)_{j \ne k})] \,.
\end{equation}
We then define the \emph{``$L^1$ influence''
of $\omega_k$ on~$f$} as the first absolute moment of $\delta_k f$:
\begin{equation} \label{eq:inf1}
	\Inf_k^{(1)}[f] := \bbE \big[ | \delta_k f | \big] 
\end{equation}
which is a quantification of ``how much $f$ depends on~$\omega_k$''.
We finally extend the definition of $\cW[f]$, see \eqref{eq:KK-original}, as the sum of squared
$L^1$ influences:
\begin{equation}\label{eq:W}
	\mathcal{W} [f] := \sum_{k} \Inf_k^{(1)} [f]^2 \,.
\end{equation}

\begin{remark}[Binary vs.\ general influences]
Definition \eqref{eq:inf1} generalises the classical influence \eqref{eq:GS-influence}
(see Lemma~\ref{lem:binary1}):
for binary variables $\omega_i$'s, see \eqref{eq:binary0}, one has the following correspondence
 for Boolean functions $f(\omega) \in \{0,1\}$
\begin{equation*}
	\Inf_k^{(1)}[f] = 2p(1-p) \, I_k[f] \,.
\end{equation*}
In particular, the definition \eqref{eq:W} of $\cW[f]$ extends the
classical one from \eqref{eq:KK-original}.
\end{remark}

\begin{remark}[$L^1$ vs.\ $L^2$ influences]
A related notion of ``$L^2$ influence'' is
\begin{equation} \label{eq:inf2}
	\Inf_k^{(2)}[f] := \bbE \big[ (\delta_k f)^2 \big] \,.
\end{equation}
For Boolean functions $f(\omega) \in \{0,1\}$
one has $\Inf_k^{(1)}[f] = 2 \, \Inf_k^{(2)}[f]$ (see Lemma~\ref{lem:boole}),
hence $L^1$ and $L^2$ influences almost coincide,
but for general functions they may differ significantly.

While the $L^2$ influence is relevant in several contexts,
see e.g.\ \cite{KKL,Tal,MOO},
it is the $L^1$ influence which plays a key role for our goals.
This was already understood in \cite{KK,ABGR14,AHT},
where the influence $I_k[f]$ for non-Boolean functions $f$ is defined as an $L^1$ norm,
see \eqref{eq:inf-quasi-gen}. 
However, definition \eqref{eq:inf-quasi-gen} still
requires binary variables $\omega_i$ and, moreover, is based on a discrete gradient which 
treats the values $x_\pm$ equally, even for $p \ne \frac{1}{2}$. This is why a ``correction factor''
$p\,(1-p)$ appears in the definition \eqref{eq:KK-original} of $\cW[f]$.

Our general definition of influence \eqref{eq:inf1} 
appears to be more suitable, since it is based on the
probabilistic gradient $\delta_k f$ which takes into account the distribution of~$\omega_k$.
\end{remark}

\subsection{General criteria for noise sensitivity}
\label{sec:general}

We are ready to state our first main results, extending
the Keller-Kindler bound \eqref{eq:KK-original} to our
general  setting, in the form \eqref{eq:vH}.

We recall that $\omega^\epsilon = (\omega^\epsilon_i)_{i\in\bbT}$ is the
configuration obtained from $\omega = (\omega_i)_{i\in\bbT}$
resampling each $\omega_i$ independently with probability~$\epsilon$,
see Definition~\ref{def:rand}.

\begin{theorem}[General noise sensitivity] \label{th:ns-gen}
Let $(\omega_i)_{i\in\bbT}$ be independent random variables
and let $f(\omega) \in L^2$ be a function satisfying Assumption~\ref{hyp:gen2}
for some $q > 2$ and $M_q < \infty$
(see e.g. Examples~\ref{ex:fin-supp} and~\ref{ex:poly-chaos}).

For any $\epsilon \in (0,1)$, there is an exponent $\gamma_{\epsilon,q} > 0$ such that
\begin{equation}
    \label{eq:boundepsilon+}
    \frac{\bbcov\bigl[ f(\omega^\epsilon), f(\omega) \bigr]}{\bbvar [f]} \le 
    4\, \bigg(  \frac{\mathcal{W} [f]}{\bbvar [f]} \bigg)^{\gamma_{\epsilon,q}} \,.
\end{equation}
The exponent $\gamma_{\epsilon,q}$, see \eqref{eq:gammaC}, depends on the
\emph{hypercontractivity constant}~$\eta_q$ defined in Lemma~\ref{th:hyp}, and it
satisfies $\gamma_{\epsilon,q} \sim \alpha_q \, \epsilon$ 
as $\epsilon \downarrow 0$
for an explicit $\alpha_q \in (0,\frac{1}{2})$, see \eqref{eq:alphaq}.
\end{theorem}

We can give a uniform lower bound on the exponent $\gamma_{\epsilon, q} \ge
\bar{\gamma}_{\epsilon, q, M_q} >0$ depending on the constant $M_q$ from 
Assumption~\ref{hyp:gen2}  (see \eqref{eq:lb-exponent} below).
As a consequence, we can apply the estimate \eqref{eq:boundepsilon+} to a sequence of functions 
$(f_N(\omega))_{N\in\N}$, which yields the announced
generalisation of the BKS criterion for noise sensitivity.

\begin{theorem}[General BKS criterion] \label{th:BKS-gen}
Let $(\omega_i)_{i\in\bbT}$ be independent random variables
and let $(f_N(\omega))_{N\in\N}$ be functions
satisfying Assumption~\ref{hyp:gen2} for the same $q > 2$ and $M_q < \infty$
(see e.g. Examples~\ref{ex:fin-supp} and~\ref{ex:poly-chaos}).
Also assume that $\limsup_{N\to\infty} \bbvar[f_N] < \infty$.

Then, recalling \eqref{eq:inf1}-\eqref{eq:W}, the BKS criterion for noise sensitivity applies:
\begin{equation}\label{eq:BKS-gen}
	\lim_{N\to\infty} \cW[f_N] = 0 \qquad \Longrightarrow \qquad
	\forall \epsilon > 0 \colon \quad
	\lim_{N\to\infty} \bbcov\bigl[f_N(\omega^\epsilon), f_N(\omega)\bigr] = 0 \,.
\end{equation}
\end{theorem}

The proof of Theorems~\ref{th:ns-gen}-\ref{th:BKS-gen} is presented in Section~\ref{sec:BKS}.
We exploit a \emph{chaos decomposition}
known as Efron-Stein or Hoeffding decomposition, see Proposition~\ref{th:genchaos},
which generalises the notion
\eqref{eq:fpoly} of polynomial chaos: we can write  any $f(\omega) \in L^2$ as
\begin{equation*}
	f(\omega) = \bbE[f] + \sum_{d=1}^\infty f^{(d)}(\omega) \qquad \text{with} \qquad
	f^{(d)}(\omega) = \sum_{I \subseteq \bbT \colon |I| = d} f_I(\omega) \,,
\end{equation*}
where the functions $f_I(\cdot)$ depend on the variables $(\omega_i)_{i \in I}$ and they are
orthogonal to each other.
By direct computation, see \eqref{eq:varianza} and \eqref{eq:cov-op}, we can write
\begin{equation} \label{eq:cov-op0}
	\bbvar[f] = \sum_{d=1}^\infty \|f^{(d)}\|_2^2 \,, \qquad
	\bbcov\bigl[ f(\omega^\epsilon), f(\omega) \bigr] = \sum_{d=1}^\infty
	(1-\epsilon)^d \, \|f^{(d)}\|_2^2 \,.
\end{equation}
The proof of the bound \eqref{eq:boundepsilon+} is then reduced to the
estimate of $\|f^{(d)}\|_2^2$ for $d \in \N$, which we call
the \emph{variance spectrum} of~$f$.

\begin{remark}[Noise sensitivity and variance spectrum]\label{rem:drift}
For functions $(f_N(\omega))_{N\in\N}$ with non-degenerate
variances, say $a \le \bbvar[f_N] \le b$ for some $0 < a < b < \infty$,
it follows by \eqref{eq:cov-op0} that \emph{noise
sensitivity} is equivalent to the \emph{variance spectrum drifting to infinity}:
\begin{equation}\label{eq:ns-spectrum}
	\forall \epsilon > 0 \colon \ \,
	\lim_{N\to\infty} \bbcov\bigl[f_N(\omega^\epsilon), f_N(\omega)\bigr] = 0
	\qquad \Longleftrightarrow \qquad
	\forall d \in \N \colon \ \,
	\lim_{N\to\infty} \|f_N^{(d)}\|_2^2 = 0 \,.
\end{equation}
\end{remark}

Proving Theorem~\ref{th:BKS-gen} is then reduced to
estimating $\|f^{(d)}\|_2^2$. For binary $\omega_i$'s,
this was obtained in \cite{KK} exploiting
\emph{exponential large deviations bounds}, 
which are not available in our setting, since Assumption~\ref{hyp:gen2}
only asks for a finite moment bound, see \eqref{eq:Mq}.
However, \emph{moment bounds turn out to be just as good
for our goals}: we show in  Theorem~\ref{th:main1} that
\begin{equation} \label{eq:boundd0}
 \frac{ \sum_{\ell = 1}^d \| f^{(\ell)} \|_2^2 }{\bbvar [f]} \le 
 \frac{1}{\eta_q^{2d}}   \left( \frac{\mathcal{W}
    [f]}{\bbvar [f]} \right)^{1 - \frac{2}{q}} 
  \end{equation}
where $\eta_q \in (0,1)$ is the \emph{hypercontractivity constant}
defined in the following result,
which follows by \cite[Proposition~3.16]{MOO}
and by the bound \eqref{eq:Mq} from Assumption~\ref{hyp:gen2}.

\begin{lemma}[Hypercontractivity constant]\label{th:hyp}
If Assumption~\ref{hyp:gen2} holds with $q > 2$ and $M_q < \infty$,
see \eqref{eq:Mq}, there is a constant $\eta_q \in (0,1)$ such that 
\begin{equation} \label{eq:hyper-gen}
\begin{split}
	\forall a,b \in \R  \colon \qquad \| a + \eta_q \, b \, X \|_q 
	\le \| a + b \, X \|_2 \qquad
	\begin{gathered}
	\text{for any } \,  X = g(\omega_i)  \, \text{ with} \\
	\bbE[X] = 0 \,, \ g \in \cV_i \,, \ i \in \bbT .
	\end{gathered}
\end{split}
\end{equation}
We fix $\eta_q \in (0,1)$ as the largest constant for which \eqref{eq:hyper-gen} holds, which satisfies
  \begin{equation} \label{eq:boundeta}
    \frac{1}{2 M_q \, \sqrt{q - 1}}
    \le \eta_q \le \frac{1}{\sqrt{q - 1}}  \,.
  \end{equation}
\end{lemma}

Once the key bound \eqref{eq:boundd0} is established, we can apply
the second relation in \eqref{eq:cov-op0} to deduce Theorem~\ref{th:ns-gen};
see Section~\ref{sec:BKS} for the details.

\begin{remark}\label{rem:hyperc}
Random variables $X$ satisfying the bound in \eqref{eq:hyper-gen} 
are called $(2,q,\eta_q)$-hypercontractive \cite{MOO,O}.
In the case of polynomial chaos, see Example~\ref{ex:poly-chaos},
since $\cV_i$ consists of linear functions,
it is enough to check \eqref{eq:hyper-gen}
for $X = \omega_i - \bbE[\omega_i]$.
\end{remark}

\begin{remark}
Even though there is a gap between the bounds in \eqref{eq:boundeta}, 
one can show that the hypercontractivity constant $\eta_q$ satisfies
$\lim_{q \downarrow 2} \eta_q = 1$, see \cite[Theorem~B.1]{CSZ20}.
\end{remark}

\subsection{Enhanced noise sensitivity}
\label{sec:enhanced}

The classical notion of noise sensitivity, namely
\begin{equation} \label{eq:ns-classical}
	\forall \epsilon > 0 \colon \qquad
	\lim_{N\to\infty} \, \bbcov\big[ f_N(\omega^\epsilon), f_N(\omega) \big] = 0 \,,
\end{equation}
yields the asymptotic independence of $f_N(\omega^\epsilon)$ and $f_N(\omega)$
for Boolean functions $f_N \in \{0,1\}$.
However, this is no longer true for functions that take more than two values.

For this reason, denoting by $C^\infty_b$ the space of bounded and smooth functions
with all bounded derivatives,
it is natural to investigate an \emph{enhanced noise sensitivity} condition:
\begin{equation} \label{eq:ns+}
	\forall \epsilon > 0 \colon \qquad
	\lim_{N\to\infty}
	\bbcov[ \varphi(f_N(\omega^\epsilon)), \psi(f_N(\omega))] 
	= 0	\qquad \forall \varphi, \psi \in C^\infty_b 
\end{equation}
which yields the asymptotic independence of $f_N(\omega^\epsilon)$ and $f_N(\omega)$
for general functions $f_N$ (with, say, uniformly bounded variance).
Indeed, if $f_N(\omega)$ converges in distribution as $N\to\infty$
to a random variable $Y$, then \eqref{eq:ns+} implies that, for any $\epsilon > 0$,
the pair $(f_N(\omega),
f_N(\omega^\epsilon))$ converges in distribution to $(Y, Y')$ where $Y'$ is an independent copy of~$Y$.

\smallskip

We now discuss conditions for the enhanced noise sensitivity property \eqref{eq:ns+}.
By an application of Cauchy-Schwarz, see \eqref{eq:cs},
it suffices to take $\psi = \varphi$ in \eqref{eq:ns+}, which reduces to the
\emph{``classical'' noise sensitivity condition \eqref{eq:ns-classical}
applied to the function $\varphi(f_N(\omega))$}.

It is natural to try and apply the BKS criterion \eqref{eq:BKS-gen}
from Theorem~\ref{th:BKS-gen}
to  $\varphi(f_N(\omega))$. 
We note that influences are stable under compositions by Lipschitz functions,
see \eqref{eq:infinf}:
\begin{equation} \label{eq:infinf0}
	\Inf_k^{(1)}[\varphi(f)] \le 2 \, \|\varphi'\|_\infty \, \Inf_k^{(1)}[f] \,,
\end{equation}
therefore $\cW[f_N] \to 0$ implies $\cW[\varphi(f_N)] \to 0$ for all $\varphi \in C^\infty_b$.
However, to apply Theorem~\ref{th:BKS-gen} to $\varphi(f_N(\omega))$, we should first check 
that $\varphi(f_N(\omega))$
for $\varphi \in C^\infty_b$  satisfies Assumption~\ref{hyp:gen2}.
This is not obvious knowing only that $f_N(\omega)$ satisfies Assumption~\ref{hyp:gen2}.

\smallskip

To avoid technicalities, 
we focus on the case when \emph{$(\omega_i)_{i\in\bbT}$ are i.i.d.\ random
variables which take finitely many values},
see Example~\ref{ex:fin-supp}. In this setting, Assumption~\ref{hyp:gen2} holds for 
\emph{any function in~$L^2$}, hence it also applies to $\varphi(f_N(\omega))$.
As a corollary of Theorem~\ref{th:BKS-gen}, we show that the BKS criterion
yields enhanced noise sensitivity.

\begin{theorem}[Enhanced noise sensitivity]\label{th:BKS+}
Let $(\omega_i)_{i\in\bbT}$ be i.i.d.\ random variables with finitely many values.
Then, for any functions $(f_N(\omega))_{N\in\N}$ we have
\begin{equation}\label{eq:BKS+}
	\lim_{N\to\infty} \cW[f_N] = 0 \qquad \Longrightarrow \qquad
	\begin{gathered}
	\rule{0pt}{1.2em}\forall \epsilon > 0 \,, \ \forall \varphi, \psi \in C^\infty_b \colon \\
	\lim_{N\to\infty} \bbcov\bigl[\varphi(f_N(\omega^\epsilon)), 
	\psi(f_N(\omega)) \bigr] = 0 \,.
	\end{gathered}
\end{equation}

This extends, for any $k\in\N$, to vector valued functions $f_N(\omega) = (f_N^{(1)}(\omega),
\ldots f_N^{(k)}(\omega))$:
the criterion \eqref{eq:BKS+} still applies
with $\cW[f_N] := \sum_{i=1}^k \cW[f_N^{(i)}]$ (for $\varphi, \psi: \R^k \to \R$).
\end{theorem}

The proof of this result is given in Section~\ref{sec:BKS}.
An interesting application to directed polymers and the SHF
is presented in Section~\ref{sec:appl},
where we exploit enhanced noise sensitivity to show \emph{asymptotic independence
of $f_N$ from any bounded order chaos}, see Theorem~\ref{th:SHF}.

\subsection{Optimal BKS}
\label{sec:optimal}

Our generalisation \eqref{eq:boundepsilon+} of the Keller-Kindler bound 
contains an exponent $\gamma_{\epsilon,q} > 0$ which depends
on the hypercontractivity constant $\eta_q$ from Lemma~\ref{th:hyp}.
We now present a refined bound, obtained by optimising over~$q$,
which is especially interesting when the hypercontractivity constant $\eta_q$ is explicit.

\smallskip

We consider in particular the situation when 
\emph{$\eta_q$ attains its largest possible value $1/\sqrt{q-1}$},
see \eqref{eq:boundeta}. This includes the classical setting \eqref{eq:binary0}
of binary $\omega_i$'s
with $p=\frac{1}{2}$, see \cite{Bon}, where we allow for arbitrary $f(\omega) \in L^2$.
More generally, by  Remark~\ref{rem:hyperc},
we can consider any polynomial chaos $f(\omega) \in L^2$
if we assume that the centred random variables $X = \omega_i - \bbE[\omega_i]$
are so-called $(2, q, 1/\sqrt{q-1})$-hypercontractive for any $q>2$, that is
\begin{equation} \label{eq:hyp-optimal}
	\forall q > 2 \,, \
	\forall a,b \in \R  \colon \qquad \| a + \eta_q \, b \, X \|_q 
	\le \| a + b \, X \|_2 \qquad \text{with} \quad \eta_q = \tfrac{1}{\sqrt{q-1}} \,.
\end{equation}
Let us recall some interesting distributions which satisfy \eqref{eq:hyp-optimal}.

\begin{example}[Optimal hypercontractivity]
\label{ex:hyp-expl}
The random variable $X = \omega_i - \bbE[\omega_i]$
satisfies the optimal bound \eqref{eq:hyp-optimal}
if $\omega_i$ has either of the following distributions:
\begin{itemize}
\item a binary distribution \eqref{eq:binary0} with $p=\frac{1}{2}$
\cite{Bon};
\item a Gaussian distribution $N(m,\sigma^2)$ for some $m \in \R$, $\sigma^2 > 0$ \cite{Nel};
\item a uniform distribution on some interval $(a,b) \subseteq \R$ \cite[Theorem~3.13]{MOO}.
\end{itemize}
\end{example}

The following refined bound is proved in Section~\ref{sec:BKS}.

\begin{theorem}[Refined bound]\label{th:BKS-optimal}
Let the independent random variables $(\omega_i)_{i\in\bbT}$
and the function $f(\omega) \in L^2$ satisfy Assumption~\ref{hyp:gen2}
for some $q = \bar{q}$.
Then we can bound
\begin{equation} \label{eq:SeV-opt}
	\forall \epsilon \in (0,1) \colon \qquad
	\frac{\bbcov\bigl[ f(\omega^\epsilon), f(\omega) \bigr]}{\bbvar [f]} \le
    	\bigg( \frac{\mathcal{W} [f]}{\bbvar [f]} \bigg)^{1-\frac{2}{q(\epsilon)}}  
\end{equation}
where $q(\epsilon) > 2$ is defined by
(recall $\eta_q \in (0,1)$ from Lemma~\ref{th:hyp})
\begin{equation} \label{eq:qepsilon}
	q(\epsilon) := \sup\bigl\{ q  \in (2, \bar{q}] \colon \ \eta_q^2 \ge 1-\epsilon \bigr\} \,.
\end{equation}

In the special case when $\eta_q = 1/\sqrt{q-1}$ for any $q > 2$,
we obtain the bound
\begin{equation} \label{eq:SeV}
	\forall \epsilon \in (0,1) \colon \qquad
	\frac{\bbcov\bigl[ f(\omega^\epsilon), f(\omega) \bigr]}{\bbvar [f]} \le
   	\bigg( \frac{\mathcal{W} [f]}{\bbvar [f]} \bigg)^{\frac{\varepsilon}{2-\varepsilon}}  
\end{equation}
which applies, in particular, to any function $f(\omega) \in L^2$ of
independent binary random variables  $(\omega_i)_{i\in\bbT}$ with $p=\frac{1}{2}$,
see \eqref{eq:binary0}.
More generally, \eqref{eq:SeV} holds for any polynomial chaos $f(\omega) \in L^2$
when $X = \omega_i - \bbE[\omega_i]$ satisfies \eqref{eq:hyp-optimal}
(see Example~\ref{ex:hyp-expl}). 
\end{theorem}

\begin{remark}
Beyond the special case $\eta_q = 1/\sqrt{q-1}$, we may consider 
$\eta_q \le 1 / (q-1)^{1/(2K)}$ for some $K \ge 1$. 
(This includes the case of binary random variables with
$p \in (0,1)$ for $K = 1/(4 p(1-p))$, see \cite[Lemma~6]{IZ}.)
In this setting, \eqref{eq:SeV} becomes
\begin{equation} \label{eq:SeV+}
	\forall \epsilon \in (0,1) \colon \qquad
	\frac{\bbcov\bigl[ f(\omega^\epsilon), f(\omega) \bigr]}{\bbvar [f]} \le
   	\bigg( \frac{\mathcal{W} [f]}{\bbvar [f]} \bigg)^{\frac{1-(1-\varepsilon)^K}{1+(1-\varepsilon)^K}} \,,
\end{equation}
see the proof of Theorem~\ref{th:BKS-optimal} in Section~\ref{sec:BKS}.
This bound agrees with \cite[Lemma~9]{IZ}.
\end{remark}

Note that \eqref{eq:SeV} recovers \eqref{eq:vH}, since
$\theta_\epsilon = \epsilon / (2-\epsilon)$. Remarkably,
\emph{this exponent is optimal}, i.e.\ there are functions $f(\omega)$
for which $\bbcov\bigl[ f(\omega^\epsilon), f(\omega) \bigr]$
matches the RHS of \eqref{eq:SeV} as $\cW[f] \downarrow 0$
(up to logarithmic corrections).
We show this by building an example
which generalises the ``Tribes'' function by M.~Ben-Or and N.~Linial \cite{BOL}.

\smallskip

We work in the symmetric binary case \eqref{eq:binary0} with 
$\bbP(\omega_i=1)=\bbP(\omega_i=-1)=1/2$ for $i\in\N$.
Let us fix a sequence $(a_t)_{t \in \N}$ with $a_t \in \N$ such that
\begin{equation}\label{eq:at}
	a_t = t^{\frac{1}{2}+\gamma} + O(1)
	\quad \text{for some } \gamma \in (0,\tfrac{1}{2})\,.
\end{equation}
We also require $t - a_t \in 2\N$ (for periodicity issues) 
so that $\bbP(\sum_{i=1}^t \omega_i = a_t ) > 0$.

Given $t,m \in \N$, we define a function
$f_{t,m}(\omega)=f_{t,m}(\omega_1,\dots,\omega_n)$ 
for $n = t \cdot m$ as follows.
\begin{itemize}
\item We divide the index set $\{1,\dots , n\} = \bigcup_{\ell=1}^m B_{\ell}$ in $m$ intervals (or ``tribes'') 
of length~$t$:
\begin{equation}\label{eq:B_ell}
	B_{\ell}:=\{(\ell -1)\,t +1, \dots , \ell \,t\} \, , \qquad
	\ell = 1,\ldots, m \,.
\end{equation}

\item To each interval we associate the random variable
\begin{equation}\label{eq:Y}
	Y_{\ell}(\omega):=\mathds{1}_{\big\{\sum_{i\in B_{\ell}}\omega_i = a_t \big\}} \, .
\end{equation}
Note that $(Y_{\ell}(\omega))_{\ell = 1, \ldots, m}$ are i.i.d.\ Bernoulli random variables
$\mathrm{Be}(p_t)$ with
\begin{equation} \label{eq:pt}
	p_t:= \mathbb{P}\left(\sum_{i=1}^t\omega_i = a_t \right)
	\underset{t\to\infty}{\sim} 2 \, \frac{\rme^{-\frac{t^{2\gamma}}{2}}}{\sqrt{2\pi \, t}}\,,
\end{equation}
by Gaussian estimates for the simple random walk (the factor $2$ is due to periodicity).

\item We define the \emph{Modified Tribes} function $f_{t,m}(\omega)$
to be the Boolean function equal to $1$ if $Y_{\ell}(\omega)=1$ for
at least one $\ell = 1,\ldots, m$, and $0$ otherwise:
\begin{equation}\label{eq:modified_tribes}
	f_{t,m}(\omega) :=\mathds{1}_{\{\exists \ell=1,\dots, m \colon Y_{\ell}(\omega) = 1\}}
	= 1 - \mathds{1}_{\{Y_{\ell}(\omega)=0, \, \forall \ell=1,\dots, m\}}.
\end{equation}
\end{itemize}
Since $\bbP(f_{t,m} = 1) = 1 - (1-p_t)^m$, for $m = m_t:=\lfloor 1/p_t \rfloor$
we have the convergence in distribution
\begin{equation*}
	f_t(\omega) := f_{t, m_t}(\omega) \xrightarrow[\ t\to\infty \ ]{d} \mathrm{Be}(1-\rme^{-1}) \,.
\end{equation*}

To show the optimality of the exponent $\varepsilon/(2-\varepsilon)$ in \eqref{eq:SeV},
we prove a lower bound which matches the RHS of \eqref{eq:SeV}
up to logarithmic corrections.
The proof is given in Section \ref{sec:optimality}. 

\begin{theorem}[Modified Tribes]\label{th:sharp}
For $\gamma \in (0,\frac{1}{2})$,
the \emph{Modified Tribes} function $f_t(\omega)$
satisfies the following: for any $\epsilon \in (0,1)$ there is $c_\epsilon > 0$
such that, as $t \to \infty$,
\begin{equation}\label{eq:sharp_bound}
	\frac{\bbcov \big[ f_t(\omega^{\varepsilon}), \, f_t(\omega) \big]}{\bbvar [f_t]} 
	\sim c_\epsilon \, \left( \frac{\mathcal{W}[f_t]}{\bbvar [f_t]} 
	\right)^{\frac{\varepsilon}{2-\varepsilon}}
	\left(\log \frac{\bbvar [f_t]}{\mathcal{W}[f_t]}\right)^{-\frac{1}{2\gamma \, (2-\epsilon)}} \,.
\end{equation}
In particular, the exponent $\varepsilon/(2-\varepsilon)$ in \eqref{eq:SeV} cannot be improved.
\end{theorem}

\begin{remark}
Replacing ``$= a_t$'' by ``$\ge a_t$'' in the definition
\eqref{eq:Y} of $Y_\ell$, we would obtain a \emph{monotone Boolean function} $f_t(\omega)$
which, we believe, still satisfies \eqref{eq:sharp_bound}
(with a different power of the logarithm). We keep the current definition
to make computations simpler.
\end{remark}

\subsection{Concluding remarks}
\label{sec:concluding}

We presented extensions of the BKS criterion for noise sensitivity beyond the classical
case of binary random variables $\omega_i$. Our setting, formulated in Assumption~\ref{hyp:gen2},
entails a form of hypercontractivity, see Lemma~\ref{th:hyp}.

\smallskip

The only result which is formulated under more restrictive conditions is Theorem~\ref{th:BKS+},
which ensures enhanced noise sensitivity. This is due to the fact that
Assumption~\ref{hyp:gen2} is not obviously stable under composition with smooth functions.
It would be interesting to investigate this issue, with the goal of 
extending the applicability of Theorem~\ref{th:BKS+}.

\smallskip

It would also be interesting to understand to which extent our key Assumption~\ref{hyp:gen2}
could be relaxed. We point out, however, that \emph{some form of hypercontractivity
must be required} for the general BKS criterion \eqref{eq:BKS-gen} to hold, as the following
example shows. 

\begin{example}[Lack of hypercontractivity]
Let $\omega = (\omega_i)_{i\in\N}$ be i.i.d.\ random variables uniformly
distributed in $(0,1)$. For $N\in\N$ we define
\begin{equation*}
	f_N(\omega) := \sum_{i=1}^N Y_{i,N}(\omega) \qquad
	\text{with} \quad
	Y_{i,N}(\omega) := \ind_{\{\omega_i > 1-\frac{1}{N}\}} \,.
\end{equation*}
Let us show that, even though 
$\cW[f_N] \to 0$, the family $(f_N)_{N\in\N}$ is \emph{not} noise sensitive.

Plainly $\bbvar[f_N] = N \, \frac{1}{N} \, (1-\frac{1}{N}) \to 1$ as $N\to\infty$.
Since $f_N = f_N^{(0)} + f_N^{(1)}$ has only chaos components of degree~$0$ and~$1$
(because $\delta_i \delta_j f_N = 0$ for $i\ne j$, see Proposition~\ref{th:genchaos}), we have
\begin{equation*}
	\bbcov[f_N(\omega^\epsilon), f_N(\omega)] = (1-\epsilon)\, \|f_N^{(1)}\|_2^2 
	= (1-\epsilon) \bbvar[f_N] \xrightarrow[\, N\to\infty \, ]{} 1 - \epsilon \ne 0 \,,
\end{equation*}
hence \emph{$(f_N)_{N\in\N}$ is not noise sensitive}. However, since
$\delta_i f_N = Y_{i,N} - \frac{1}{N}$, we have
\begin{equation*}
	\Inf_i^{(1)}[f_N] = \bbE\bigl[ |\delta_i f_N | \bigr]
	= \tfrac{2}{N}\bigl( 1-\tfrac{1}{N}\bigr) \quad\ \Longrightarrow \quad\
	\cW[f_N] = \sum_{i=1}^N \Inf_i^{(1)}[f_N]^2 \le \tfrac{4}{N}
	\xrightarrow[\, N\to\infty \, ]{}  0 \,,
\end{equation*}
which shows that \emph{the general BKS criterion \eqref{eq:BKS-gen} does not apply}.

This is not in contrast with our Theorem~\ref{th:BKS-gen} because 
\emph{the functions $f_N$
fail to satisfy Assumption~\ref{hyp:gen2}} with the same $q > 2$ and $M_q < \infty$: indeed,
we must have
\begin{equation*}
	M_q \ge \frac{\| Y_{i,N} - \bbE[Y_{i,N}] \|_q }{\| Y_{i,N} - \bbE[Y_{i,N}] \|_2 }
	\sim N^{\frac{1}{2}-\frac{1}{q}} \xrightarrow[\, N\to\infty \, ]{}  \infty\,.
\end{equation*}

We can also consider the monotone Boolean functions 
$g_N(\omega) := \ind_{\{f_N(\omega) \ge 1\}}$. 
One can show that these functions
are also \emph{not} noise sensitive,
even though $\cW[g_N] \to 0$.
\end{example}

\section{Application to 2D directed polymers and SHF}
\label{sec:appl}

We present an application of our noise sensitivity results to the so-called
\emph{directed polymer in random environment}. This is a much studied model in statistical mechanics 
and probability theory, both as a prototype of \emph{disordered system} and 
because of its link with the (multiplicative)
\emph{Stochastic Heat Equation} and \emph{Kardar-Parisi-Zhang (KPZ) 
equation} and universality class. We refer to 
\cite{Comets,Niko,CSZ24} for more details and references.

\smallskip

We focus here on the case of space dimension~$2$.
In a critical regime of vanishing disorder strength, directed polymer partition functions,
averaged over the starting point on the diffusive scale,
converge to a universal limit,
named in \cite{CSZ23} the \emph{Critical 2D Stochastic Heat Flow (SHF)}
(see Theorem~\ref{thm:CSZ23} below).
An axiomatic characterisation of the SHF was recently provided by L.-C. Tsai \cite{T24},
who also established continuity in time as a measure valued process.
It was shown in the same paper that
the SHF also arises as the limit of solutions of the 2D Stochastic Heat Equation with
space-regularised white noise, as the regularisation is removed and
the noise intensity is suitably rescaled.

\smallskip

We prove here the following novel results.

\begin{itemize}
\item In Theorem~\ref{thm:noise_sensitivity_polymer}, 
we establish \emph{enhanced
noise sensitivity} for the directed polymer partition functions
converging to the SHF.

\item In Theorem~\ref{th:SHF} we deduce \emph{the independence of the SHF from the white noise}
arising from the scaling limit of the environment.

\end{itemize}
In the setting of the Stochastic Heat Equation, a similar result
was obtained simultaneously and independently in \cite{GT},
as a corollary of their main result that the SHF is a black noise
(see also \cite{HP24} for a related result on the directed landscape).

\smallskip

Many properties related to the SHF have been investigated, including moments
\cite{BC95,CSZ20,GQT21,C24,LZ24,GN}, comparison with Gaussian multiplicative 
chaos \cite{CSZ23b,CM24,CT}, path measures and
flow property \cite{CM24},
singularity and regularity \cite{CSZ25}, 
strict positivity \cite{N25b,CT}, weak disorder fluctuations \cite{CCR25,CSZ25},
strong disorder and large scale behavior \cite{BCT,CT}.
See also \cite{N25a,C25} for recent progress on a martingale description of the SHF.
We refer to the recent lecture notes \cite{CSZ24} for an extended discussion
and further references.

\subsection{Setting}

Let $\mathbb{T}=\mathbb{N}\times \mathbb{Z}^d$ and $\omega = (\omega (n,x))_{(n,x)\in \mathbb{T}}$ be a family of i.i.d.\ random variables under the law $\bbP$, called
\emph{environment} or \emph{disorder}, with
zero mean, unit variance and finite exponential moments:
\begin{equation*}
	\bbE[\omega(n,x)] = 0 \,, \qquad
	\bbE[\omega(n,x)^2] = 1 \,, \qquad
	\lambda(\beta) := \log \bbE[\rme^{\beta \omega(n,x)}] <\infty \ \forall \beta > 0 \,.
\end{equation*}
Let $(S_n)_{n\in\mathbb{N}}$ be the symmetric simple random walk on $\mathbb{Z}^d$, with law and expectation $\P, \E$.

Given a disorder realisation $\omega$, a system size $N\in \mathbb{N}$, and
inverse temperature (or disorder strength) $\beta>0$, the \emph{directed polymer} is the random probability law for $S$ given by
\begin{equation*}
	\dd \P_{N}^{\omega,\beta}(S) 
	:= \frac{\rme^{\sum_{n=1}^N (\beta\omega(n,S_n)-\lambda(\beta))}}{Z_{N}^{\omega,\beta}}
	\dd \P(S) \,,
\end{equation*}
where $Z_{N}^{\omega,\beta}$ is the normalising constant called \emph{partition function}.

For every dimension $d$ there exists a critical value $\beta_c = \beta_c(d) \ge 0$ 
such that for $\beta \le \beta_c$ the behaviour of the polymer path is diffusive, while for $\beta > \beta_c$ a localised behavior emerges. The critical value is known to be $\beta_c>0$ for $d\geq 3$, while $\beta_c=0$ for $d=1,2$. Remarkable progress on 
the understanding of the the critical point for $d \ge 3$
was only obtained recently, see \cite{J23,HL24,L25} and references therein.

\smallskip

In the recent years, there has been much focus on studying the scaling properties of the model
in a finer window around the critical point $\beta_c = 0$ 
when the dimension equals one or two. To this end, one can rescale
the disorder strength as the volume grows by letting $\beta=\beta_N\to0$ appropiately, in order to obtain non-degenerate limits for $Z_{N,\beta_N}^{\omega}$. The case of spatial dimension one was first treated by Alberts, Khanin and Quastel \cite{AKQ14} who showed that the right scaling is
$\beta_N \sim \hat\beta \, N^{-1/4}$.

\smallskip

In dimension $d=2$ the right scaling is
$\beta_N \sim \hat\beta / \sqrt{R_N}$, where $R_N = \sum_{n=1}^N \P(S_n = S'_n)
\sim \frac{1}{\pi} \, \log N$ (we denote by $S'$ an independent copy of~$S$).
It was first shown in \cite{CSZ17} that the model exhibits a phase trasition on this finer scale
with critical point $\hat\beta_c = 1$. 

The behaviour at the critical point $\hat\beta = 1$
is more subtle: the partition function $Z_N$ starting from a point
converges to zero in distribution, while its expectation is one and its second and higher moments diverge. This intermittent behavior suggests to average the partition function in space in order to obtain a meaningful limit.

Let us first state precisely the scaling of $\beta_N$ around $\hat\beta_c = 1$,
known as \emph{critical window}:
\begin{equation}\label{eq:sigma}
	\sigma_N^2:=\rme^{\lambda(2\beta_N)-2\lambda(\beta_N)}-1 
	=\frac{1}{R_N}\left( 1+\frac{\vartheta + o(1)}{\log N}\right), \quad \vartheta\in\mathbb{R}.
\end{equation}
Given two continuous and compactly supported test functions $g,h \in C_c(\mathbb{R}^2)$ and two times $0 \le s < t < \infty$, consider the \emph{averaged} partition function
\begin{equation}\label{eq:partition_function}
	Z_{N;s,t}^{\omega,\beta}(g,h)
	:= \frac{1}{N} \sum_{x_0, y_0\in\mathbb{Z}^2}
	g(\tfrac{x_0}{\sqrt{N}}) \,
	\E\pigl[\rme^{ H^{\omega,\beta}_{(Ns,Nt]}(S,\, \omega)} \, \ind_{\{S_{Nt}=y_0\}}
	\, \pig| \, S_{Ns} = x_0 \pigr] \, h(\tfrac{y_0}{\sqrt{N}}) \,,
\end{equation}
where
\begin{equation*}
	H_{(0,N]}^{\omega,\beta}(S,\omega)=\sum_{n=1}^N
	\{\beta \, \omega(n,S_n) - \lambda(\beta)\}
	=\sum_{(n,x)\in\mathbb{T}} \{ \beta \, \omega(n,x) - \lambda(\beta)\}\mathds{1}_{S_n=x}	\,.
\end{equation*}
The following is the main result of \cite{CSZ23}.

\begin{theorem}[Directed Polymer and Stochastic Heat Flow \cite{CSZ23}]\label{thm:CSZ23}
Fix $\beta_N$ in the critical window \eqref{eq:sigma}
for some $\theta\in\R$.
For any $g,h \in C_c(\mathbb{R}^2)$ and $0 \le s < t < \infty$,
the partition functions in \eqref{eq:partition_function}
with $\beta=\beta_N$ converge in distribution 
to a non trivial limit:
\begin{equation}\label{eq:CSZ23}
	Z_{N;s,t}^{\omega,\beta_N}(g,h)
	\ \xrightarrow[\ N\to \infty \ ]{d} \
	\mathcal{Z}_{s,t}^\theta(g,h) \ :=\iint_{\mathbb{R}^2\times\R^2} g (x)
	\, h(y) \, \mathcal{Z}_{s,t}^\theta(\dd x, \dd y) \, ,
\end{equation}
where $\mathcal{Z}^\theta = ( \mathcal{Z}_{s,t}^\theta(\dd x, \dd y) )_{0 \le s < t < \infty}$
is a stochastic process 
of random measures on $\R^2 \times \R^2$, called the critical 2D Stochastic Heat Flow.
The convergence in \eqref{eq:CSZ23} holds jointly over $s,t$, $g, h$
in the sense of finite-dimensional distributions (f.d.d.).
\end{theorem}

\subsection{Enhanced noise sensitivity for directed polymers}

In our first main result for the directed polymer model,
we establish enhanced noise sensitivity for 
the sequence of  rescaled partition functions appearing in \eqref{eq:CSZ23}.

\begin{theorem}[Enhanced noise sensitivity for 2D directed polymers]\label{thm:noise_sensitivity_polymer}
Consider $\beta_N$ in the critical window \eqref{eq:sigma}
with $\theta\in\R$.
Given $g,h \in C_c(\mathbb{R}^2)$ and $0 \le s < t < \infty$,
the partition function $Z_{N;s,t}^{\omega,\beta_N}(g,h)$ in \eqref{eq:CSZ23}
is, for every $N\in\N$, a function of the environment:
\begin{equation} \label{eq:fZ}
	f_N(\omega) := Z_{N;s,t}^{\omega,\beta_N}(g,h) \,.
\end{equation}
This sequence of functions is noise sensitive, namely
\begin{equation} \label{eq:CovZ}
	\forall\, \varepsilon >0 \colon \qquad
	\bbcov [f_N(\omega^{\varepsilon}), f_N(\omega)]
	=\bbcov \pigl[Z_{N;s,t}^{\omega^\epsilon,\beta_N}(g,h) \,, \,
	Z_{N;s,t}^{\omega,\beta_N}(g,h) \pigr]\xrightarrow[N\to\infty]{}0 \,.
\end{equation}

\smallskip

If we further assume that the disorder variables $\omega(n,z)$ take finitely many values,
enhanced noise sensitivity holds: for any $k\in\N$, given
$g_i,h_i \in C_c(\mathbb{R}^2)$ and $0 \le s_i < t_i < \infty$ for $i=1,\ldots, k$,
defining the vector-valued function $f_N(\omega) \in \R^k$ by
\begin{equation*}
	f_N(\omega) = (f_N^{(1)}(\omega), \, \ldots,\, f_N^{(k)}(\omega))
	:= \pigl( Z_{N;s_i,t_i}^{\omega,\beta_N}(g_i,h_i) \pigr)_{i=1,\ldots, k} \, ,
\end{equation*}
the following asymptotic independence property holds:
\begin{equation}\label{eq:CovphiZ}
	\forall\, \varepsilon >0 \,, \ \forall \varphi, \psi \in C^\infty_b(\R^k \to \R) \colon \qquad
	\bbcov \bigl[\varphi(f_N(\omega^{\varepsilon})), \, \psi(f_N(\omega)) \bigr]
	\xrightarrow[N\to\infty]{}0 \,.
\end{equation}
\end{theorem}

\begin{remark}[Polynomial chaos]
The partition function $f_N(\omega)$ in \eqref{eq:fZ} satisfies our Assumption~\ref{hyp:gen2}
because, for any $(n,z) \in \N \times \Z^2$, it depends on $\omega(n,z)$ as a linear function of
the ``tiled'' random variables $\zeta(n,z)$ defined by
\begin{equation*}
	\zeta(n,z) := \rme^{\beta_N \omega(n,z) - \lambda(\beta_N)}-1 \,,
\end{equation*}
see \eqref{eq:quasi-poly} below, hence \eqref{eq:belongs} holds with $\cV_{(n,z)} =
\{a \, \zeta(n,z)  + b \colon a, b \in \R\}$.
Equivalently, we can view $f_N(\omega)$ as a polynomial chaos in the variables $\zeta(n,z)$,
see Example~\ref{ex:poly-chaos}.
\end{remark}

We prove Theorem~\ref{thm:noise_sensitivity_polymer} exploiting our general noise sensitivity results,
more precisely we deduce \eqref{eq:CovZ} from Theorem~\ref{th:BKS-gen}
and \eqref{eq:CovphiZ} from Theorem~\ref{th:BKS+}.
To this purpose, it suffices to show that
the function $f_N(\omega)$ defined in \eqref{eq:fZ} satisfies the general BKS criterion
\begin{equation}\label{eq:BKS-Z}
	\cW[f_N] = \sum_{(n,z) \in \N \times \Z^2} \mathrm{Inf}_{(n,z)}^{(1)}[f_N]^2
	\xrightarrow[\ N\to\infty \ ]{} 0 \, ,
\end{equation}
for any fixed $g,h \in C_c(\mathbb{R}^2)$ and $0 \le s < t < \infty$
(we recall that the influence $\mathrm{Inf}_k^{(1)}[f]$ was defined in \eqref{eq:inf1}).
This follows from the following computation, proved in Subsection~\ref{sec:influZ}.
For easy of notation, we focus on the simple  case when $s=0, t=1$.

\begin{proposition}[Influences for directed polymer]\label{thm:influences_polymer}
For $\beta_N$ in \eqref{eq:sigma},
consider the partition function $f_N(\omega) := Z_{N;0,1}^{\omega,\beta_N}(g,h)$
in \eqref{eq:partition_function} with $s=0$, $t=1$ and $g, h \in C_c(\R^2)$.

For any $(n,z)\in \N \times \Z^2$ we have
\begin{equation}\label{eq:influences_polymer}
	\Inf_{(n,z)}^{(1)}[f_N]
	\le \ind_{\{1,\ldots,N\}}(n) \; \frac{\sigma_N}{N} \sum_{x_0, y_0\in\mathbb{Z}^2}
	|g(\tfrac{x_0}{\sqrt{N}})| \,
	q_n(z-x_0) \, q_{N-n}(y_0 - z) \, |h(\tfrac{y_0}{\sqrt{N}}) |
\end{equation}
with $\sigma_N$ from \eqref{eq:sigma},
where $q_n(z-x)=\P(S_n=z\, | \, S_0=x)$ is the random walk kernel.
\end{proposition}

\begin{proof}[Proof of Theorem~\ref{thm:noise_sensitivity_polymer}]
It suffices to show that \eqref{eq:BKS-Z} holds. We will prove that
\begin{equation}\label{eq:loginfl}
	\cW[f_N] = O\bigg(\frac{1}{\log N}\bigg) \,.
\end{equation}

We recall that $g, h \in C_c(\R^2)$. For simplicity,
let $g$ be supported in the unit ball $\{|\cdot| \le 1\}$. We can bound
$|h(\cdot)| \le \|h\|_\infty$ in \eqref{eq:influences_polymer}, after which the sum
over $y_0 \in \Z^2$ gives~$1$. We can also restrict the sum to
$|x_0| \le \sqrt{N}$ and bound $|g(\cdot)| \le \|g\|_\infty$. We then obtain
\begin{equation*}
\begin{split}
	\cW[f_N]
	& = \sum_{(n,z) \in \N \times \Z^2} \Inf_{(n,z)}^{(1)}[f_N]^2 \\
	&\le \|g\|_\infty^2 \, \|h\|_\infty^2 \, \frac{\sigma_N^2}{N^2} \, \sum_{n=1}^N \sum_{z\in\Z^2} 
	\sumtwo{x_0, x_0' \in \Z^2}{|x_0|, |x_0'| \le \sqrt{N}}
	q_n(z-x_0) \, 
	q_n(z-x_0') \\
	&= \|g\|_\infty^2 \, \|h\|_\infty^2 \, \frac{\sigma_N^2}{N^2} \, \sum_{n=1}^N
	\sumtwo{x_0, x_0' \in \Z^2}{|x_0|, |x_0'| \le \sqrt{N}} \,
	q_{2n}(x_0 - x_0')
	= O(\sigma_N^2) \,,
\end{split}
\end{equation*}
where the last equality holds because $\sum_{x_0 \in \Z^2} q_{2n}(x_0 - x_0') = 1$
and the sums over~$x_0'$ and~$n$ give $O(N^2)$.
Since $\sigma_N^2 = O(1/R_N) = O(1/\log N)$, see \eqref{eq:sigma},
we have proved \eqref{eq:loginfl}.
\end{proof}

\begin{remark}
In the recent work by Y. Gu and T. Komorowski \cite{GK}, a proof 
of classical (non ehnanced) noise sensitivity for the SHE and the Schr\"odinger equation 
is given, either in one dimension with white noise, or in higher dimensions with space-colored noise.
\end{remark}

\subsection{Asymptotic independence of SHF and white noise}

By the central limit theorem,
the disorder field $\omega = (\omega(n,z))$, suitably rescaled,
converges in distribution to (space-time) white noise $\xi = \xi(t,x)$ on $[0,\infty) \times \R^2$.
The latter is defined as the centred Gaussian generalised field with covariance
$\bbcov[\xi(s,y), \xi(t,x)] = \delta(t-s) \, \delta(x-y)$.

\smallskip

For a precise statement, let us consider test functions $\rho \in C_c([0,\infty) \times \R^2)$.
Evaluating the white noise $\xi (\rho )$ on $\rho$, which formally corresponds
to $\int_{[0,\infty) \times \R^2} \xi(t,x) \, \rho(t,x) \, \dd t \, \dd x$,
one gets a genuine centred Gaussian process $(\xi(\rho))_{\rho\in C_c}$ with
$\bbcov[\xi (\rho ), \xi (\tilde\rho )] = \langle \rho, \tilde{\rho} \rangle_{L^2}$.

For $N\in\N$, define an approximation $\xi_N$ of the white noise 
by the rescaled disorder~$\omega$:
\begin{equation*}
	\xi_N(\rho) := \frac{1}{N} \, \sum_{(n,z) \in \N \times \Z^2} \rho(\tfrac{n}{N}, \tfrac{z}{\sqrt{N}})
	\, \omega(n,z) \,.
\end{equation*}
It is easy to check that, as $N\to\infty$, one has the convergence in distribution
$\xi_N(\rho) \to \xi(\rho)$ jointly over $\rho \in C_c([0,\infty) \times \R^2)$
in the finite-dimensional distributions sense.

\smallskip

Summarising, we have the two convergences
\begin{gather}
	\label{eq:convZ}
	\cZ_N := \big( Z_{N;s,t}^{\omega,\beta_N}(g,h) \big)_{s\le t, g,h \in C_c(\R^2)}
	\xrightarrow[\ N\to\infty \ ]{d}
	\cZ^\theta = \big( \mathcal{Z}_{s,t}^\theta(f,g) )_{s\le t, g,h \in C_c(\R^2)} \, , \\
	\label{eq:convxi}
	\xi_N = (\xi_N(\rho))_{\rho \in C_c([0,\infty) \times \R^2)}
	\xrightarrow[\ N\to\infty \ ]{d}
	\xi = (\xi(\rho))_{\rho \in C_c([0,\infty) \times \R^2)} \, ,
\end{gather}
both in the finite-dimensional distributions sense (for simplicity: see Remark~\ref{rem:topology} below).
It is then natural to ask for the \emph{joint convergence of the pair $(\cZ_N, \xi_N)$}.

For each $N\in\N$, we remark that $\cZ_N$ 
is a \emph{function of $\omega$}, hence it is also a \emph{function of $\xi_N$}.
Nevertheless, remarkably,
\emph{this dependence is fully lost as $N\to\infty$}:
our next result shows that the SHF $\cZ^\theta$ which arises
as the limit of $\cZ_N$ is \emph{independent} of the
white noise $\xi$ which arises as the limit of $\xi_N$.

\begin{theorem}[Independence of SHF and white noise]\label{th:SHF}
Assume that the disorder variables $\omega(n,z)$ take finitely many values.
As $N\to\infty$, we have the joint convergence in distribution
$(\cZ_N, \xi_N) \to (\cZ^\theta, \xi)$, in the f.d.d. sense,
where $\cZ^\theta$ and $\xi$ are independent.
\end{theorem}

\begin{proof}
We use the shorthand $\chi = (s,t;g,h)$ so that we may view
$\cZ_N$ and $\cZ^\theta$ as processes indexed by $\chi$.
It suffices to show the following: for any $k, \ell\in\N$,
given any $\vec{\chi} = (\chi_1, \ldots, \chi_k)$ and 
$\vec{\rho} = (\rho_1, \ldots, \rho_\ell)$,
we have the asymptotic independence of the random vectors
\begin{equation*}
	\cZ_N(\vec{\chi}) := \big( \cZ_N(\chi_i)\big)_{1\le i \le k} \qquad
	\text{and} \qquad
	\xi_N(\vec{\rho}) := \big( \xi_N(\rho_i) \big)_{1\le i \le \ell} \,,
\end{equation*}
that is for a suitable class of test functions $\varphi: \R^k \to \R$ and $\psi: \R^\ell \to \R$ we have
\begin{equation} \label{eq:finalcov}
	\lim_{N\to\infty} \bbcov\big[ \varphi(\cZ_N(\vec{\chi})) \,, \,
	\psi(\xi_N(\vec{\rho})) \big] = 0 \,.
\end{equation}
We are going to take $\varphi \in C^\infty_b$, while for $\psi$ we take
arbitrary polynomials (they are enough to determine convergence in distribution
to Gaussian random variables).

\smallskip

Each random variable $\xi_N(\rho_i)$ is a linear function of $\omega$,
hence its chaos decomposition \eqref{eq:f-chaos} only contains terms of degree~$1$.
Since $\psi$ is a polynomial, say of degree~$d$, it follows that
\emph{the chaos decomposition of $\psi(\xi_N(\vec{\rho}))$ only contains terms
of degree $\le d$}.

On the other hand, by the enhanced noise sensitivity property \eqref{eq:CovphiZ},
it follows that \emph{the variance spectrum of $\varphi(\cZ_N(\vec{\chi}))$ drifts to infinity}
for any $\varphi \in C^\infty_b$, see Remark~\ref{rem:drift},
namely the contribution to $\bbvar[\varphi(\cZ_N(\vec{\chi}))]$ given by chaos of degree $\le d$ vanishes
as $N\to\infty$.

Looking back at \eqref{eq:finalcov},
we can replace $\varphi(\cZ_N(\vec{\chi}))$ by
\begin{equation*}
	\varphi(\cZ_N(\vec{\chi}))^{(>d)} := \varphi(\cZ_N(\vec{\chi}))
	- \varphi(\cZ_N(\vec{\chi}))^{(\le d)} \, ,
\end{equation*}
see \eqref{eq:fled},
i.e.\ we can remove the terms of degree $\le d$, up to a negligible error in~$L^2$.
After this modification, the covariance vanishes because $\varphi(\cZ_N(\vec{\chi}))^{(>d)}$
is orthogonal to $\psi(\xi_N(\vec{\rho}))$, which only contains
terms of degree $\le d$. The proof is complete.
\end{proof}

\begin{remark}[Stronger topologies] \label{rem:topology}
Both convergences \eqref{eq:convZ} and \eqref{eq:convxi} are known to hold beyond f.d.d.,
under stronger topologies (e.g.\ in the space of continuous measure-valued processes 
for $\cZ_N \to \cZ^\theta$, see \cite{T24}, and in a negative Hölder space for $\xi_N \to \xi$).

Since tightness for the pair $(\cZ_N, \xi_N)$ follows by tightness for the marginals ---
and since f.d.d.'s determine the law ---
Theorem~\ref{th:SHF}  yields the joint convergence 
$(\cZ_N, \xi_N) \to (\cZ^\theta, \xi)$ under the
corresponding product topology, with $\cZ^\theta$ and $\xi$ 
independent.
\end{remark}

\subsection{Proof of Theorem~\ref{thm:influences_polymer}}
\label{sec:influZ}
Let us decompose the expectation appearing in \eqref{eq:partition_function}
on the event $\{S_n = z\}$ and on its complement $\{S_n \ne z\}$.
Writing $\E_{x_0}$ for $\E[\,\cdot\,|\,S_0 = x_0]$, using the Markov
property of the random walk as well as the additivity of the Hamiltonian, we obtain
\begin{equation} \label{eq:quasi-poly}
\begin{split}
	\E_{x_0} \pigl[ \rme^{H^{\omega,\beta_N}_{(0,N]}(S,\, \omega)} 
	& \, \ind_{\{S_{N}=y_0\}} \pigr]
	= a_1 \, \rme^{\beta_N \omega(n,z) - \lambda(\beta_N)} \, a_2 \,+\, b   \\
	\text{with} \qquad a_1 &:= \E_{x_0}\pigl[\rme^{ H^{\omega,\beta_N}_{(0,n-1]}(S,\, \omega)}
	\, \ind_{\{S_n = z\}} \pigr] \,, \\
	a_2 & := \E\pigl[\rme^{ H^{\omega,\beta_N}_{(n+1,N]}(S,\, \omega)} \, \ind_{\{S_{N}=y_0\}}
	\,\pig|\, S_n = z \pigr] \,, \\
	b & := \E_{x_0}\pigl[\rme^{ H^{\omega,\beta_N}_{(0,N]}(S,\, \omega)}
	\, \ind_{\{S_n \ne z\}} \, \ind_{\{S_{N}=y_0\}} \pigr] \,.
\end{split}
\end{equation}
Note that the terms $a_1, a_2, b$ \emph{do not depend on $\omega(n,z)$}.
In particular, the term $b$ vanishes when we apply the operator
$\delta_{(n,z)} f = f - \bbE_{(n,z)} f$. It follows that
\begin{equation*}
\begin{split}
	& \delta_{(n,z)} \, \E_{x_0}\pigl[ 
	\rme^{ H^{\omega,\beta_N}_{(0,N]}(S,\, \omega)} \, \ind_{\{S_{N}=y_0\}} \pigr] 
	= a_1 \, \big( \rme^{\beta_N \omega(n,z) - \lambda(\beta_N)} - 1 \big) \,
	a_2 \,.
\end{split}	
\end{equation*}
We next observe that, by \eqref{eq:sigma},
\begin{equation*}
	\bbE[|\rme^{\beta_N \omega(n,z) - \lambda(\beta_N)} - 1|]
	\le \sqrt{\bbE[(\rme^{\beta_N \omega(n,z) - \lambda(\beta_N)} - 1)^2] } =
	\sqrt{\rme^{\lambda(2\beta_N)-2\lambda(\beta_N)}-1} = \sigma_N \, ,
\end{equation*}
while by Fubini, since $\bbE[\rme^{H_{(a,b]}^{\omega,\beta}}] = 1$, we have
\begin{equation*}
	\bbE [ a_1 ] = \E_{x_0}\big[ \ind_{\{S_n = z\}} \big] =
	q_n(z-x_0) \,, \qquad
	\bbE [a_2] = \E[ \ind_{\{S_{N}=y_0\}}
	\,|\, S_n = z ] = q_{N-n}(y_0-z) \,.
\end{equation*}
Plugging these computations into \eqref{eq:partition_function}, 
since $\Inf_{(n,z)}^{(1)}[f_N] = \bbE[|\delta_{(n,z)} f_N|]$, we obtain \eqref{eq:influences_polymer}.

\section{General setting}
\label{sec:setting}

We give in this section the definition of the objects of our interest.

\subsection{Random variables}

We recall that $\omega = (\omega_i)_{i \in \mathbb{T}}$ denote independent
random variables, defined on a probability space $(\Omega, \mathcal{A}, \mathbb{P})$
and with values in measurable spaces $(E_i, \cE_i)$, with laws $\mu_i$, see \eqref{eq:Ei}.
Whenever we write $f(\omega)$ we imply that $f: \bigtimes_{i\in \bbT} E_i \to \R$
is measurable with respect to $\bigotimes_{i\in\bbT} \cE_i$,
so that $f(\omega)$ is a random variable defined on~$\Omega$.

It is convenient to work on a general probability space
---rather than on the canonical space $(\Omega,\cA,\bbP) = (\bigtimes_{i\in \bbT} E_i,
\otimes_{i\in\bbT} \cE_i, \otimes_{i\in\bbT} \mu_i)$---
to allow for extra randomness.

\begin{definition}[$\epsilon$-randomisation] \label{def:rand}
For $\varepsilon \in [0, 1]$,
we denote by $\omega^{\varepsilon} =
(\omega^{\varepsilon}_i)_{i \in \mathbb{T}}$  the modified family
where $\omega^{\varepsilon}_i = \omega_i$
with probability $1-\epsilon$, while $\omega^{\varepsilon}_i$
is independently resampled with probability~$\epsilon$. More explicitly, we define
\begin{equation}\label{eq:tildeomega}
  \omega_i^\epsilon := \omega_i \,\ind_{\{ U_i > \varepsilon \}} +
  \tilde{\omega}_i \,\ind_{\{ U_i \le \varepsilon \}} \, ,
\end{equation}
where $\tilde{\omega} = (\tilde{\omega}_i)_{i \in \mathbb{T}}$ is an
independent copy of $\omega = (\omega_i)_{i \in \mathbb{T}}$ and $U =
(U_i)_{i \in \mathbb{T}}$ are independent random variables 
(also independent of $\omega, \tilde\omega$)
uniformly distributed in $(0, 1)$.
\end{definition}

We often omit $\omega$ and write
$\bbE[f]$, $\bbvar[f]$, etc. We
denote by $\langle \cdot, \cdot \rangle$ the $L^2$ scalar product:
\begin{equation*}
	\langle f, g \rangle := \bbE[f(\omega) \, g(\omega)] \,.
\end{equation*}
We denote by $\cF_I$ the $\sigma$-algebra generated by the
random variables $\omega_i$ for $i\in I \subseteq \bbT$:
\begin{equation}\label{eq:FI}
	\cF_I := \sigma(\omega_i \colon i \in I) \,.
\end{equation}
We denote by $L^2(\cF_I)$ the subspace of $L^2$ random variables which
are $\cF_I$-measurable,
i.e.\ functions of $(\omega_i)_{i\in I}$.
We recall that $\bbE_i[\,\cdot\,] = \bbE[\,\cdot\,|\cF_{\bbT \setminus \{i\}}]$,
see \eqref{eq:deltak}.

\subsection{General chaos decomposition}
\label{sec:ES}

Every function $f(\omega) \in L^2$
admits an orthogonal \emph{chaos decomposition} similar to the definition
of a polynomial chaos, see \eqref{eq:fpoly},
where monomials $\hat{f}(i_1, \ldots, i_d) \, \omega_{i_1} \cdots \omega_{i_d}$ are
replaced by suitable orthogonal functions $f_I(\omega)$.
This is also called \emph{Hoeffding or Efron-Stein decomposition}.

\begin{proposition}[Chaos decomposition]\label{th:genchaos}
If $\omega = (\omega_i)_{i\in\bbT}$ are independent random variables,
any function $f(\omega) \in L^2$ can be written as the $L^2$ convergent series
\begin{equation}\label{eq:f-chaos}
	f(\omega) 
	=\, f^{(0)}
	+ \sum_{d = 1}^\infty \, f^{(d)}(\omega) \qquad
	\text{with} \quad
	\begin{cases}
	f^{(0)} = f_\emptyset = \bbE[f] \,, \\
	\rule{0pt}{1.6em}\displaystyle
	f^{(d)}(\omega) = \sum_{I \subseteq \bbT
	\colon  |I| = d} f_I(\omega)\,,
	\end{cases}
\end{equation}
for a unique choice of functions $f_I(\omega) \in L^2$,
labelled by finite subsets $I \subseteq \bbT$, such that
\begin{equation}\label{eq:fI}
	f_I \in L^2(\cF_I) \qquad \text{and} \qquad
	\bbE_k[f_I] = 0 \quad \forall k \in I \,.
\end{equation}
Explicitly,
recalling the operator $\delta_i\, g := g - \bbE_i[g]$ from \eqref{eq:deltak},
we can write
\begin{equation}\label{eq:fIdef}
	\text{for } I = \{i_1, \ldots, i_d\} \colon \qquad
	f_I  =\delta_{i_1} \cdots \,\delta_{i_d} \, \tilde{f}_I
	\qquad \text{with} \quad \tilde{f}_I := \bbE[f \,|\, \cF_I] \,.
\end{equation}
\end{proposition}

We refer to \cite[Section~8.3]{O} for a discussion
(where $f_I, \tilde f_I$ are denoted by $f^{=I}, f^{\subseteq I}$).
We provide a compact proof of Proposition~\ref{th:genchaos}
in Appendix~\ref{app:ES}. 
One can also give a hands-on construction of the functions $f_I$'s
by fixing a basis of~$L^2$, see Lemma~\ref{th:tensor}
and Remark~\ref{rem:ViL2}.

\begin{remark}
The first property in \eqref{eq:fI} means that $f_I$ is a function of $\omega_I = (\omega_i)_{i\in I}$. 
The second property in \eqref{eq:fI} means that, if $f_I$ is non zero, then ``$f_I$ depends on
every $\omega_k$ for $k\in I$'' (if $f_I$ does not depend on some $\omega_k$, i.e.\ if it is constant w.r.t.\ $\omega_k$, then $f_I = \bbE_k[f_I] = 0$).
\end{remark}

The properties \eqref{eq:fI} imply that, for all $k\in\bbT$ and $I \subseteq \bbT$,
\begin{equation} \label{eq:fImore}
	\delta_k f_I := f_I - \bbE_k[f_I] = \begin{cases}
	f_I & \text{if } k \in I \,, \\
	0 & \text{if } k \not\in I \,,
	\end{cases} \qquad
	\bbE[f_I | \cF_{J}] =
	\begin{cases}
	f_I & \text{if } I \subseteq J \,, \\
	0 & \text{if } I \not\subseteq J \,.
	\end{cases}
\end{equation}
Then it follows by the chaos decomposition \eqref{eq:f-chaos} that
\begin{equation} \label{eq:deltakEJf}
	\delta_k f = \sum_{I \supseteq \{k\} } f_I \,, \qquad
	\bbE[f|\cF_J] = \sum_{I \subseteq J} f_I \,.
\end{equation}
The second property in \eqref{eq:fImore} implies
that the functions $f_I$'s are orthogonal:
\begin{equation}\label{eq:orthogonal}
	\langle f_I, f_J \rangle = 0 \qquad \text{for any }I \ne J \,,
\end{equation}
therefore by \eqref{eq:f-chaos} we obtain
\begin{equation} \label{eq:varianza}
	\bbvar[f] 
	\,= \sum_{d = 1}^\infty \|f^{(d)}\|_2^2 
	\,= \sum_{d = 1}^\infty \, \sum_{I \subseteq \bbT :\, |I| = d} \| f_I \|_2^2 \,.
\end{equation}
Similarly, for any functions $f(\omega), g(\omega) \in L^2$,
\begin{equation} \label{eq:covarianza}
	\bbcov[f,g] 
	\,= \sum_{d = 1}^\infty \langle f^{(d)}, g^{(d)} \rangle
	\,= \sum_{d = 1}^\infty \, \sum_{I \subseteq \bbT :\, |I| = d} \langle f_I,
	g_I \rangle \,.
\end{equation}

Given $f(\omega) \in L^2$ with chaos decomposition \eqref{eq:f-chaos},
we define its projection $f^{(\le d)}(\omega)$ on chaos of order up to~$d$ by
\begin{equation} \label{eq:fled}
	f^{(\le d)}(\omega) := \sum_{\ell = 0}^d f^{(\ell)}(\omega)
	= \sum_{I \subseteq \bbT \colon  |I| \le d} f_I(\omega) \,.
\end{equation}
(An important result by Bourgain \cite[Proposition~6]{B},
see also \cite[Theorem~10.39]{O},
ensures that such projections are bounded in~$L^q$ for $1 < q < \infty$.)

\subsection{Noise operator}

Given $\eta \in [0,1]$, we define a \emph{noise operator} $T^\eta : L^2 \to L^2$
acting on functions $f(\omega) \in L^2$ with chaos decomposition \eqref{eq:f-chaos}-\eqref{eq:fI} 
as follows:
\begin{equation} \label{eq:Teta}
	T^\eta f (\omega) \,:=\, \bbE[f] + 
	\sum_{d = 1}^\infty \eta^d \, f^{(d)}(\omega)
	\,=\, \bbE[f] + 
	\sum_{d = 1}^\infty \eta^d \sum_{I \subseteq \bbT :\, |I| = d} f_I(\omega) \,.
\end{equation}

Recalling the definition \eqref{eq:tildeomega} of $\omega^\epsilon = (\omega^\epsilon_i)_{i\in\bbT}$,
we introduce the conditional expectation
\begin{equation*}
	\bbE[\,\cdot\,|\,\omega] :=
	\bbE[\,\cdot\,|\,\sigma(\omega_i)_{i\in\bbT}] \, ,
\end{equation*}
i.e.\ we integrate out the $U_i$'s and $\tilde\omega_i$'s in \eqref{eq:tildeomega}. 
In the next result, we compute the covariance between $f(\omega^\epsilon)$ and $f(\omega)$
explicitly in terms of the operator $T^\eta$.

\begin{lemma}
For any $\epsilon \in (0,1)$ we have
\begin{equation} \label{eq:fT}
	\bbE[ f(\omega^\epsilon) \,|\, \omega] = T^{1-\epsilon} f(\omega) \,,
\end{equation}
hence, recalling the chaos decomposition \eqref{eq:f-chaos}-\eqref{eq:fI},
for any $f(\omega), g(\omega) \in L^2$ we can write
\begin{equation} \label{eq:cov-op}
\begin{split}
	\bbcov[f(\omega^\epsilon), g(\omega)] 
	&\,=\, \sum_{d = 1}^\infty (1-\epsilon)^d \, \langle f^{(d)} , g^{(d)}\rangle \\
	& \,=\, \sum_{d = 1}^\infty \epsilon \, (1-\epsilon)^d \,
	\bbcov[f^{(\le d)}, g^{(\le d)}] \,,
\end{split}
\end{equation}
where we note that $\bbcov[f^{(\le d)}, g^{(\le d)}] = \sum_{\ell = 1}^d \langle f^{(\ell)} , 
g^{(\ell)}\rangle$ (see \eqref{eq:covarianza} and \eqref{eq:fled}). Moreover,
\begin{equation}\label{eq:cs}
	\bbcov[f(\omega^\epsilon), g(\omega)] 
	\le \sqrt{\bbcov[f(\omega^\epsilon), f(\omega)] 
	\cdot \bbcov[g(\omega^\epsilon), g(\omega)] } \,.
\end{equation}
\end{lemma}

\begin{proof}
The first equality in \eqref{eq:cov-op} follows by \eqref{eq:fT} and \eqref{eq:f-chaos} since
\begin{equation*}
	\bbE[f(\omega^\epsilon) \, g(\omega)]  =
	\bbE[ \, \bbE[f(\omega^\epsilon) | \omega] \, g(\omega)]  = \langle T^{1-\epsilon} f , g \rangle \,,
\end{equation*}
while the second equality in \eqref{eq:cov-op} holds by summation by parts.
The inequality \eqref{eq:cs} follows by \eqref{eq:cov-op} applying Cauchy-Schwarz.

It remains to prove \eqref{eq:fT}, for which it suffices to show that
\begin{equation} \label{eq:sensitive-computation}
	\bbE[ f_I(\omega^\epsilon) \,|\, \omega] = (1-\epsilon)^{|I|} \, f_I(\omega)
	\qquad \forall I \subseteq \bbT \,.
\end{equation}
For any fixed $k\in I$, 
writing $\omega^\epsilon = (\omega^\epsilon_j)_{j \ne k} \cup \{\omega^\epsilon_k\}$,
we note that the event $\{U_k \le \epsilon\}$ gives a null contribution
to \eqref{eq:sensitive-computation}, because
the variable $\omega^\epsilon_k = \tilde\omega_k$ can be integrated out:
\begin{equation*}
\begin{split}
	\bbE[ f_I(\omega^\epsilon) \, \ind_{\{U_k \le \epsilon\}} \,|\, \omega]
	&= \bbE\big[ f_I\big( (\omega^\epsilon_j)_{j \ne k} \cup \{\tilde\omega_k\} \big) 
	\, \ind_{\{U_k \le \epsilon\}} \,\big|\, \omega\big] \\
	&= \bbE\big[ \, \bbE_k[ f_I ]\big( (\omega^\epsilon_j)_{j \ne k} \big) 
	\, \ind_{\{U_k \le \epsilon\}} \,\big|\, \omega\big] = 0 \,,
\end{split}
\end{equation*}
by the second property in \eqref{eq:fI}. We then restrict the LHS
of \eqref{eq:sensitive-computation} to the event $\bigcap_{k\in I}\{U_k > \epsilon\}$,
on which we have $f_I(\omega^\epsilon) = f_I(\omega)$ by the first property in \eqref{eq:fI}.
This proves \eqref{eq:sensitive-computation}.
\end{proof}

\subsection{Hypercontractivity}
\label{sec:hyper}

The noise operator $T^\eta$ enjoys a fundamental \emph{hypercontractivity property}
when applied to functions $f(\omega) \in L^2$ which satisfy Assumption~\ref{hyp:gen2}. 
We recall that the hypercontractivity constant $\eta_q \in (0,1)$ 
is defined in Lemma~\ref{th:hyp}.

\begin{theorem}[General hypercontractivity]\label{th:hyper}
Let the independent random variables $(\omega_i)_{i\in\bbT}$ and
the function $f(\omega) \in L^2$ satisfy Assumption~\ref{hyp:gen2}
for some $q > 2$.

Let $f'(\omega)$ be any linear combination $f'$
of components $f_I$ from \eqref{eq:f-chaos}:
\begin{equation}\label{eq:f'}
	f'(\omega) = \sum_{I \subseteq \bbT} \alpha_I \, f_I(\omega)
	\qquad \text{with} \quad \|f'\|_2^2 = \sum_{I \subseteq \bbT} \alpha_I^2 \, \|f_I\|_2^2 < \infty \,.
\end{equation}
Defining $\eta_q \in (0,1)$ as in Lemma~\ref{th:hyp},
the noise operator $T^{\eta_q}$ from \eqref{eq:Teta} satisfies
\begin{equation} \label{eq:hyper0}
	\| T^{\eta_q} f' \|_q \le  \| f' \|_2 \,.
\end{equation}
In particular, if $f' = f'_{d}$ only contains terms of degree up to~$d$, we can bound
\begin{equation} \label{eq:hyper}
 	\| f'_{d} \|_q \le \frac{1}{\eta_q^d} \,  \| f'_{d} \|_2
	\qquad \text{for } \
	f'_{d} = \sum_{I \subseteq \bbT \colon |I| \le d} \alpha_I \, f_I \,.
\end{equation}
\end{theorem}

We show in Appendix~\ref{app:hyper}
that  Theorem~\ref{th:hyper} is a slight extension of 
results from \cite{MOO} (which focus on finite $\bbT$
and finite dimensional $\cV_i$).
To this purpose, we show in Lemma~\ref{th:tensor} that functions $f(\omega) \in L^2$
which satisfy Assumption~\ref{hyp:gen2}~\eqref{it:a} can be characterised as multi-linear polynomials
with respect to suitable \emph{ensembles} \cite[Definition~3.1]{MOO}.

\begin{remark}[Conjugated hypercontractivity]
The hypercontractivity bound \eqref{eq:hyper0} from $L^2$ to $L^q$ implies a similar
hypercontractivity from $L^p$ to $L^2$ for the conjugated exponent $p= \frac{q}{q-1}$
(so that $\frac{1}{p} + \frac{1}{q} = 1$).
Indeed, by Hölder we can bound
\begin{equation*}
	\| T^{\eta_q} f' \|_2 = \langle T^{\eta_q} f', T^{\eta_q} f' \rangle
	= \langle f', T^{\eta_q} ( T^{\eta_q} f') \rangle \le  \| f' \|_p
	\, \| T^{\eta_q} ( T^{\eta_q} f') \|_q,
\end{equation*}
hence applying \eqref{eq:hyper0} (with $T^{\eta_q} f'$ in place of $f'$) we obtain
\begin{equation} \label{eq:hyper0p}
	\| T^{\eta_q} f' \|_2 \le  \| f' \|_p \,.
\end{equation}
\end{remark}

\section{Proof of noise sensitivity criteria}
\label{sec:BKS}

In this section we prove Theorem~\ref{th:BKS-optimal} and
Theorem~\ref{th:ns-gen},
from which we then deduce Theorems~\ref{th:BKS-gen}, \ref{th:BKS+}.
We fix a family
$(\omega_i)_{i\in\bbT}$ of independent random variables as in \eqref{eq:Ei}.

\subsection{Preparation}

We first state a basic interpolation bound.

\begin{remark}[Interpolation bound]
For any $g\in L^2$ we can bound
\begin{equation}\label{eq:interpo}
	\forall p \in (1,2) , \ q \in (2,\infty) \
	\text{ with } \ \tfrac{1}{p} + \tfrac{1}{q} = 1 \colon \qquad
	\| g \|_{p} \le \| g \|_1 ^{1 - \frac{2}{q}} \,
	\| g \|_2^{\frac{2}{q}}
\end{equation}
(just write $p  = \alpha \, 1 + (1 - \alpha) \, 2$ for $\alpha = 2 - p = \frac{q-2}{q-1} \in (0, 1)$
and apply H{\"o}lder).
\end{remark}

We recall the chaos decomposition \eqref{eq:f-chaos} of a function $f(\omega) \in L^2$,
which yields the expansion \eqref{eq:varianza} for $\bbvar[f]$,
and the projection $f^{(\le d)}$ on chaos of order $\le d$, see \eqref{eq:fled}.

The core of our proof is the next result which bounds 
the contribution of~$\|f^{(\le d)}\|_2^2$
in terms of the sum of squared $L^1$ influences $\cW[f]$, see \eqref{eq:W}.
This is close in spirit to \cite[Lemma~6]{KK}, but we use moment bounds rather than
large deviations. A key tool is the hypecontractivity estimate
\eqref{eq:hyper} from Theorem~\ref{th:hyper}.

\begin{theorem}[Key bound]\label{th:main1}
Let $(\omega_i)_{i\in\bbT}$ be independent random variables
and let $f(\omega) \in L^2$ satisfy Assumption~\ref{hyp:gen2}
for some $q > 2$.
Define $\eta_q \in (0,1)$ as in Lemma~\ref{th:hyp}.

Then we can bound
\begin{equation} \label{eq:boundd}
	\text{for } \ \epsilon \ge 1 - \eta_q^2 \colon \qquad
	 \frac{ \bbcov [f(\omega^{\varepsilon}),f(\omega)]}{\bbvar [f]} 
	 \le \left( \frac{\mathcal{W}[f]}{\bbvar [f]} \right)^{1 - \frac{2}{q}} 
  \end{equation}
which yields
\begin{equation}\label{eq:bounddd}
	\forall d \in \N \colon \qquad 
	\frac{ \| f^{(\le d)} \|_2^2 }{\bbvar [f]}
	= \frac{ \sum_{\ell = 1}^{d} \| f^{(\ell)} \|_2^2 }{\bbvar [f]}
	\le \frac{1}{\eta_q^{2d}} \, \left( \frac{\mathcal{W}
    [f]}{\bbvar [f]} \right)^{1 - \frac{2}{q}} \,.
\end{equation}
\end{theorem}

\begin{proof}
Without loss of generality, \emph{we assume that the index set $\bbT$
is totally ordered} (e.g.\ via a correspondence with $\N$).
For finite $I \subseteq \bbT$ we can thus consider $\max(I) \in \bbT$ and we note that
$\max(I) = k$ if and only if $I = \{k\} \cup J$ with $J < k$
(i.e.\ $j < k$ for all $j\in J$).

Fix $q > 2$ from Assumption~\ref{hyp:gen2} and $\epsilon \ge 1 - \eta_q^2$.
We can write, by the first line of \eqref{eq:cov-op},
\begin{equation} \label{eq:passo0}
\begin{split}
	\bbcov[f(\omega^\epsilon), f(\omega)]
	\le  \sum_{I\subseteq \bbT } \eta_q^{2|I|} \|f_I\|_2^2 = \sum_{k\in\bbT}
	\Bigg\{ \sum_{I\subseteq \bbT \colon \max(I) = k} \eta_q^{2|I|}\|f_I\|_2^2 \Bigg\} \,.
\end{split}
\end{equation}
The term in bracket can be identified with $\|T^{\eta_q}g_k\|_2^2$ (recall \eqref{eq:Teta})
where we define the function
\begin{equation} \label{eq:gk}
	g_{k}(\omega)
	:= \sum_{I\subseteq \bbT \colon \max(I) = k} f_I(\omega) \,,
\end{equation}
hence we can rewrite \eqref{eq:passo0} as
	\begin{equation} \label{eq:passo01}
	\bbcov[f(\omega^\epsilon), f(\omega)] \, = \,
	\sum_{k \in \bbT} \, \|T^{\eta_q}g_{k}\|_2^2 \,.
\end{equation}

The function $g_k$ is connected to the influence $\Inf_k^{(1)}[f]$, see \eqref{eq:inf1}:
indeed,
introducing the $\sigma$-algebra $\cF_{\le k} := \sigma(\omega(i) \colon i \le k)$,
we can write by \eqref{eq:fImore}-\eqref{eq:deltakEJf}
\begin{equation}\label{eq:altgk}
	g_{k}(\omega) = \bbE\bigg[ \, \sum_{I \subseteq \bbT \colon I \ni k}
	f_I \,\bigg|\, \cF_{\le k} \bigg]
	= \bbE[ \, \delta_k f \,|\, \cF_{\le k} ] \, .
\end{equation}

It follows by \eqref{eq:gk} that
\begin{equation*}
	\|T^{\eta_q}g_{k}\|_2^2 = \big\langle T^{\eta_q}g_{k} \,, \, T^{\eta_q} g_{k} \big\rangle = \big\langle g_{k} \,, \, T^{\eta_q} (T^{\eta_q} g_{k}) \big\rangle\,.
\end{equation*}
Let $p = \frac{q}{q-1}$ be the conjugate exponent of~$q$.
Applying Hölder, the interpolation bound \eqref{eq:interpo}
and the hypercontractivity bound \eqref{eq:hyper0}
(note that $g_{k}$ satisfies condition \eqref{eq:f'}), we get
\begin{equation*}
	\|T^{\eta_q} g_{k}\|_2^2 \le \big\| g_{k} \big\|_p \, \big\| T^{\eta_q} (T^{\eta_q}g_{k}) \big\|_q
	\le \, 
	\big\| g_{k} \big\|_1^{1-\frac{2}{q}} \, \big\| g_{k} \big\|_2^{\frac{2}{q}}
	\, \big\| T^{\eta_q} g_{k} \big\|_2 \,.
\end{equation*}
After simplifying the last factor $\big\| T^{\eta_q} g_{k} \big\|_2$ with one power
from the LHS, we can plug this estimate
into \eqref{eq:passo01}. Applying Hölder again, we arrive at
\begin{equation}\label{eq:passo1bis}
	\bbcov [f(\omega^{\varepsilon}),f(\omega)] \,\le\,
	\bigg( \sum_{k \in \bbT} \| g_k \|_1^2 \bigg)^{1-\frac{2}{q}}
	\, \bigg( \sum_{k \in \bbT} \| g_k \|_2^2 \bigg)^{\frac{2}{q}}  \,. 
\end{equation}

To complete the proof of \eqref{eq:boundd}, it suffices to show that
\begin{equation}\label{eq:tsh}
	\sum_{k \in \bbT} \| g_k \|_1^2 \le \cW[f] \,, \qquad
	\sum_{k \in \bbT} \| g_k \|_2^2 \le \bbvar[f] \,.
\end{equation}
The first relation in \eqref{eq:tsh} is a consequence of \eqref{eq:altgk} and \eqref{eq:inf1}, since
\begin{equation*}
	\| g_k \|_1 = \bbE\big[ \, \big| \, \bbE[ \, \delta_k f \,|\, \cF_{\le k} ]
	\, \big| \, \big]
	\le \bbE\big[ \, \bbE[ \, | \delta_k f | \,|\, \cF_{\le k} ] \, \big]
	= \bbE[ \, | \delta_k f | \,] \,.
\end{equation*}
The second relation in \eqref{eq:tsh} follows directly by \eqref{eq:altgk}
and \eqref{eq:varianza}, since
\begin{equation*}
	\sum_{k \in \bbT} \|g_k\|_2^2 
	= \sum_{k \in \bbT} \sum_{I \subseteq \bbT \colon \max(I) = k}
	\| f_{I} \|_2^2
	\le \sum_{I \subseteq \bbT \colon |I| \ge 1} \|f_I\|_2^2 = \bbvar[f] \,.
\end{equation*}
The proof is completed.
\end{proof}

\smallskip
\subsection{Proof of Theorem~\ref{th:BKS-optimal}}
The first bound \eqref{eq:SeV-opt} follows directly from Theorem~\ref{th:main1},
because the constraint $\epsilon \ge 1-\eta_q^2$ in \eqref{eq:boundd}
is the same as $\eta_q^2 \ge 1-\epsilon$ appearing in the definition \eqref{eq:qepsilon}
of  $q(\varepsilon)$. Then \eqref{eq:SeV-opt} is recovered from \eqref{eq:boundd} with $q=q(\varepsilon)$.

Let us now consider the special case of $\eta_q = 1 / (q-1)^{1/(2K)}$ for some $K \ge 1$.
A direct computation gives
\begin{equation*}
	\eta_{q}^2 \ge 1-\varepsilon \qquad \text{ if and only if } 
	\qquad q \le q(\epsilon) := 1 + \frac{1}{(1-\epsilon)^K} \,.
\end{equation*}
Applying \eqref{eq:SeV-opt} with $q(\varepsilon)$ as above leads to \eqref{eq:SeV+},
which for $K=1$ reduces to \eqref{eq:SeV}.
\qed

\smallskip
\subsection{Proof of Theorem~\ref{th:ns-gen}}
\label{sec:ns-gen}
We fix $q \in (2,\infty)$ as in Assumption~\ref{hyp:gen2}
and $\eta_q \in (0,1)$ as in Lemma~\ref{th:hyp}.
Let us define
\begin{equation}\label{eq:condepsilon}
	\bar{\epsilon}_q := 1 - \eta_q^2 > 0 \,,
\end{equation}
We will prove \eqref{eq:boundepsilon+} for the following exponent $\gamma_{\epsilon,q}$:
\begin{equation}\label{eq:gammaC}
	\gamma_{\epsilon,q} :=
	\begin{cases}
	 \big( 1 - \tfrac{2}{q}\big) \,
	\frac{\log\frac{1}{1-\epsilon}}{\log \eta_q^{- 2}}
	& \text{if } \epsilon \le \frac{1}{2}\bar{\epsilon}_q \,, \\
	\gamma_{\frac{1}{2}\bar{\epsilon}_q,q}
	& \text{if } \epsilon > \frac{1}{2}\bar{\epsilon}_q \,.
	\end{cases}
\end{equation}
\emph{It suffices to show that \eqref{eq:boundepsilon+} holds for 
$\epsilon \in (0, \frac{1}{2} \bar\epsilon_q]$}:
indeed
$\bbcov[f(\omega^\epsilon), f(\omega)]$ is decreasing in~$\epsilon$, see \eqref{eq:cov-op}, 
hence the bound extends to $\epsilon > \frac{1}{2} \bar\epsilon_q$ 
since we set $\gamma_{\epsilon,q} := \gamma_{\bar\epsilon_q/2, q}$.

\begin{remark}
One could obtain a sharper bound
for $\epsilon > \bar\epsilon_q/2$ refining the proof below, but we omit the details,
since we are mostly interested in small~$\epsilon$.
\end{remark}

\begin{remark}
Note that as $\epsilon \downarrow 0$
\begin{equation}\label{eq:alphaq}
      \gamma_{\epsilon,q} \,\sim\, \alpha_q \, \epsilon \qquad \text{with} \qquad
	\alpha_q := \frac{1 - \frac{2}{q}}{\log \eta_q^{-2}}
      \,\in\, \left( 0, \tfrac{1}{2} \right) \,.
\end{equation}
To prove that $\alpha_q < \frac{1}{2}$,
note that hypercontractivity constant satisfies 
$\eta_q \le \frac{1}{\sqrt{q - 1}}$, see \eqref{eq:boundeta}, hence
\begin{equation*}
	\alpha_q \le f(q) := \frac{1-\frac{2}{q}}{\log(q-1)} \,.
\end{equation*}
Since $\lim_{q \downarrow 2} f(q) = \frac{1}{2}$,
it suffices to show that $f(\cdot)$ is strictly decreasing for $q > 2$.
We have
  \begin{equation*}
	f'(q) 	= \frac{2 \, (q-1) \, \log(q-1) - q\,(q-2)}{q^2 \, (q-1) \, (\log(q-1))^2} 
\end{equation*}
and we note that the numerator is strictly negative for for $q>2$, as it
vanishes for $q = 2$ and its derivative equals $2 \log(q-1)  + 4 - 2q < 0$
(by $\log x < x-1$ for $x > 1$).
\end{remark}

Henceforth we assume that $0 < \epsilon \le \frac{1}{2} \bar\epsilon_q$.
We introduce the shorthand
\begin{equation} \label{eq:cdelta}
	\delta := 
	\left( \frac{\mathcal{W}[f]}{\bbvar [f]} \right)^{1 - \frac{2}{q}} \,.
\end{equation}
Since $\eta_q^2 = 1- \bar{\epsilon}_q$,
we can bound $\sum_{\ell=1}^d\|f^{(\ell)}\|_2^2 
\le (1-\bar{\epsilon}_q)^{-d} \, \delta
\, \bbvar[f]$ by Theorem~\ref{th:main1}, see \eqref{eq:bounddd}. 
We use \eqref{eq:cov-op} to estimate 
$\bbcov[f(\omega^\epsilon), f(\omega)]$:
for any $\bar{d} \in \N$, recalling \eqref{eq:varianza}, we get
\begin{equation} \label{eq:covbound}
\begin{split}
	\bbcov[f(\omega^\epsilon), f(\omega)] 
	&\le \sum_{d=1}^{\bar{d}} \epsilon \,
	(1-\varepsilon)^d \, \bigg(\sum_{\ell=1}^d\|f^{(\ell)}\|_2^2\bigg)
	+ (1-\varepsilon)^{\bar{d}+1} \, \bbvar [f] \\
	&\le \Bigg\{ \delta \, \epsilon \, \sum_{d=1}^{\bar{d}} 
	\bigg(\frac{1-\varepsilon}{1-\bar{\epsilon}_q}\bigg)^d \, 
	+ (1-\varepsilon)^{\bar{d}+1} \Bigg\}\,\bbvar [f] \,,
\end{split}
\end{equation}
For $\epsilon \le \frac{1}{2} \bar\epsilon_q$ 
we have $\rho := \frac{1-\varepsilon}{1-\bar{\epsilon}_q} \ge \frac{1-\epsilon}{1-2\epsilon} \ge 1 + \epsilon > 1$, hence $\sum_{\ell = 1}^{\bar d} \rho^\ell
\le \frac{\rho^{\bar d + 1} }{\rho - 1}
\le \frac{1}{\epsilon}\, \frac{(1-\epsilon)^{\bar d + 1}}{(1-\bar{\epsilon}_q)^{\bar d + 1}}$ which yields
\begin{equation*}
	\frac{\bbcov[f(\omega^\epsilon), f(\omega)]}{\bbvar[f]}
	\le (1-\varepsilon)^{\bar{d}+1} \Bigg\{ 
	\frac{\delta}{(1-\bar{\epsilon}_q)^{\bar d + 1}}
	+ 1 \Bigg\} \,.
\end{equation*}
Plainly, this estimate holds also for $\bar{d} = 0$
(since the LHS is at most $1-\epsilon$, see \eqref{eq:cov-op}).

We first assume that $\frac{\delta}{1-\bar{\epsilon}_q} \le 1$.
Let $\bar{d} \in \N_0 = \{0,1,\ldots\}$ be the largest integer for which
\begin{equation*}
	(1-\bar{\epsilon}_q)^{\bar d} \ge \frac{\delta}{1-\bar{\epsilon}_q} \, ,
\end{equation*}
that is
\begin{equation*}
	\bar{d}=\bigg\lfloor \frac{\log \frac{1-\bar{\epsilon}_q}{\delta}}{\log \frac{1}{1-\bar{\epsilon}_q}} \bigg\rfloor \geq 0.
\end{equation*}
For such $\bar{d}$ we can estimate
\begin{equation*}
	\frac{\bbcov[f(\omega^\epsilon), f(\omega)]}{\bbvar[f]} \le
	2 \, (1-\varepsilon)^{\bar{d}+1}
	\le 2 \, (1-\varepsilon)^{\frac{\log \frac{1-\bar{\epsilon}_q}{\delta}}{\log \frac{1}{1-\bar{\epsilon}_q}}}
	= 2 \, \bigg(\frac{\delta}{1-\bar{\epsilon}_q}
	\bigg)^{\frac{\log\frac{1}{1-\epsilon}}{\log \frac{1}{1-\bar{\epsilon}_q}}} \,,
\end{equation*}
and this bound holds also if $\frac{\delta}{1-\bar{\epsilon}_q} > 1$
(the LHS is always at most~$1$).
We finally note that 
\begin{equation*}
	2 \, \bigg(\frac{\delta}{1-\bar{\epsilon}_q}
	\bigg)^{\frac{\log\frac{1}{1-\epsilon}}{\log \frac{1}{1-\bar{\epsilon}_q}}}
	= \frac{2}{1-\epsilon} \, 
	\big(\delta)^{\frac{\log\frac{1}{1-\epsilon}}{\log \frac{1}{1-\bar{\epsilon}_q}}} \le 
	4 \, \left( 
		\frac{\mathcal{W}[f]}{\bbvar [f]} \right)^{\gamma_{\epsilon,q}} \, ,
\end{equation*}
for $\epsilon \le \frac{1}{2} \bar\epsilon_q \le \frac{1}{2}$,
by definition of $\delta$ and $\gamma_{\epsilon,q}$,
see \eqref{eq:cdelta} and \eqref{eq:gammaC}.
The proof of \eqref{eq:boundepsilon+} is complete.\qed

\smallskip
\subsection{Proof of Theorem~\ref{th:BKS-gen}}
The exponent $\gamma_{\epsilon,q}$ appearing in \eqref{eq:boundepsilon+},
defined in \eqref{eq:gammaC},
contains the hypercontractive constant $\eta_q$ from Lemma~\ref{th:hyp},
which depends on the vector spaces $\cV_i$ in Assumption~\ref{hyp:gen2}.
However, by \eqref{eq:gammaC} and \eqref{eq:boundeta}, we can bound
\begin{equation} \label{eq:lb-exponent}
	\gamma_{\epsilon,q} \ge 
	\bar{\gamma}_{\epsilon,q} :=
	\big( 1 - \tfrac{2}{q}\big) \,
	\frac{\log\frac{1}{1-\epsilon}}{\log (4 M_q^2 (q-1))} \,,
\end{equation}
and note that $\bar{\gamma}_{\epsilon,q}$ only depends on $\epsilon, q$ and~$M_q$.
In particular, given functions $(f_N(\omega))_{N\in\N}$
which satisfy Assumption~\ref{hyp:gen2} for the same $q > 2$ and $M_q < \infty$,
we can apply the bound \eqref{eq:boundepsilon+} with $\gamma_{\epsilon,q}$ 
replaced by $\bar{\gamma}_{\epsilon,q}$, which proves
the general BKS criterion \eqref{eq:BKS-gen}.\qed

\smallskip
\subsection{Proof of Theorem~\ref{th:BKS+}}
We follow the arguments sketched in the discussion before Theorem~\ref{th:BKS+}.
Let us first prove the bound \eqref{eq:infinf0}, recalling \eqref{eq:inf1}:
for $\varphi \in C^\infty_b$
\begin{equation} \label{eq:infinf}
\begin{split}
	\Inf_k^{(1)}[\varphi(f)] &= \bbE \bigl[ \bigl| \varphi(f) - \bbE_k[\varphi(f)] \bigr| \bigr] \\
	& \le \bbE \bigl[ \bigl| \varphi(f) - \varphi(\bbE_k[f]) \bigr| \bigr] 
	+ \bbE \bigl[ \bigl| \bbE_k[\varphi(\bbE_k[f]) - \varphi(f)]  \bigr| \bigr] \\
	& \le 2 \, \|\varphi'\|_\infty \, \bbE\bigl[ |f - \bbE_k[f]| \bigr] 
	= 2 \, \|\varphi'\|_\infty \, \Inf_k^{(1)}[f] \,.
\end{split}	
\end{equation}
This easily extends to a vector of functions $f(\omega) = (f^{(1)}(\omega),
\ldots f^{(k)}(\omega))$ with $\varphi: \R^k \to \R$:
\begin{equation*}
	\Inf_k^{(1)}[\varphi(f)] \le 2 \, \|\varphi'\|_\infty \, \sum_{i=1}^k \Inf_k^{(1)}[f^{(i)}]
	\qquad \text{with} \quad
	\|\varphi'\|_\infty := \max_{1 \le i \le k} \|\partial_i \varphi\|_\infty \,.
\end{equation*}
In particular, by \eqref{eq:W},
\emph{condition $\cW[f] := \sum_{i=1}^k \cW[f_N^{(i)}] \to 0$
implies $\cW[\varphi(f_N)] \to 0$}.

For i.i.d.\ random variables $(\omega_i)_{i\in\bbT}$ with finitely many values,
we can apply Theorem~\ref{th:BKS-gen} to \emph{any functions in $L^2$ with bounded variance},
in particular to $\varphi(f_N)$ for $\varphi \in C^\infty_b$. Then condition
$\cW[f_N] \to 0$ --- which implies $\cW[\varphi(f_N)] \to 0$ --- yields
$\bbcov\bigl[\varphi(f_N(\omega^\epsilon)), \varphi(f_N(\omega))\bigr] \to 0$
for any $\epsilon > 0$ and $\varphi \in C^\infty_b$, see \eqref{eq:BKS-gen}. Finally,
by an application of Cauchy-Schwarz, see \eqref{eq:cs}, we obtain
$\bbcov\bigl[\varphi(f_N(\omega^\epsilon)), \psi(f_N(\omega))\bigr] \to 0$
for any $\varphi, \psi \in C^\infty_b$, which proves \eqref{eq:BKS+}.\qed

\section{Modified Tribes function}
\label{sec:optimality}

In this section we prove Theorem~\ref{th:sharp}.
We first consider a general
class of ``tribes''-like functions, for which computations are more transparent.
We will then specialise to the specific function from \eqref{eq:modified_tribes}
to reach the optimal exponent 
$\varepsilon/(2-\varepsilon)$ in \eqref{eq:SeV}.

\subsection{Tribes-like functions}

For $t\in\N$ let $A_t = A_t(\omega_1,\dots , \omega_t)$ be an
event depending on $t$ variables, that is invariant under permutations of the $\omega_i$'s and has probability
\begin{equation*}
	p_t := \bbP (A_t) \to 0 \qquad \text{as } t\to\infty\,.
\end{equation*}
We set
\begin{equation*}
	m_t := \lfloor 1/p_t \rfloor 
\end{equation*}
and consider the intervals $(B_{\ell})_{\ell =1,\dots , m_t}$ in \eqref{eq:B_ell}.
Denoting by $\omega_{B_{\ell}}$ the collection $(\omega_i)_{i\in B_{\ell}}$, we define
\begin{equation}\label{eq:general_Y}
	Y_{\ell}(\omega):= \mathds{1}_{A_{t}(\omega_{B_{\ell}})}
	\qquad \text{and} \qquad f_t(\omega):=\mathds{1}_{\{
	\exists \ell=1,\dots, m_t \colon Y_{\ell}(\omega)=1 \}} \,.
\end{equation}
Note that $(Y_{\ell}(\omega))_{\ell}$ are i.i.d.\ Bernoulli random variables of parameter $p_t$, hence
\begin{equation}\label{eq:var_tribes}
	\mathbb{E}[f_t]= 1 - (1-p_t)^{m_t}\xrightarrow[t\to\infty]{} 1 - \rme^{-1} \qquad \text{and} 
	\qquad \bbvar [f_t]\xrightarrow[t\to\infty]{} \rme^{-1}(1-\rme^{-1}).
\end{equation}

\begin{remark}
The definition of $f_t$ is coherent with \eqref{eq:Y} when 
$A_t=\big\{ \sum_{i=1}^t \omega_i = a_t\big \}$.
The original Tribes function \cite{BL85} 
corresponds to $A_t = \{\omega_i = 1 \ \forall i=1,\ldots, t\}$.
\end{remark}

We recall that $\omega^{\varepsilon}$ denotes the $\varepsilon$-randomisation of $\omega$,
see Definition~\ref{def:rand}. Let us define
\begin{equation}\label{eq:q_t_def}
q_{t,\epsilon}:=\bbP (Y_1(\omega^{\varepsilon})=1\, | \, Y_1(\omega)=1)= \bbP (A_t(\omega^{\varepsilon}) \, | \, A_t(\omega)) \,.
\end{equation}
The covariance between $f_t(\omega ^{\varepsilon})$ and $f_t(\omega)$ has an explicit asymptotic behavior described in the following Lemma, whose proof is given in Subsection~\ref{sec:proof_cond1}.

\begin{lemma}\label{lem:cond1}
Suppose that $p_t=o(q_{t,\epsilon})$ as $t\to \infty$, for any $\epsilon \in (0,1)$.
Then, as $t\to\infty$,
\begin{equation}\label{eq:cond1}
	\bbcov [f_t(\omega^{\varepsilon}),f_t(\omega)]
	\sim \rme^{-2} \, q_{t,\epsilon} \,.
\end{equation}
\end{lemma}

We next consider the influence of $\omega_k$ on $Y_\ell(\omega)$
(recall the classical definition \eqref{eq:GS-influence} of influence).
Since $Y_\ell(\omega)$ is invariant under permutations of the $\omega_k$'s,
it is enough to focus on
\begin{equation}\label{eq:rt}
	r_t := I_1[Y_1] = \bbP (Y_1(\omega^1_{+})\neq Y_1(\omega^1_{-}))
	=\bbP (A_t(\omega^1_{+})\cap A_t(\omega^1_{-})= \varnothing) \,,
\end{equation}
where we recall that $\omega_\pm^1$ is the configuration $\omega$ where we fix $\omega_1 = \pm 1$.
We use $r_t$ to express the influence of $\omega_k$ on~$f_t$:
recalling \eqref{eq:modified_tribes}, a direct computation shows that as $t\to\infty$
\begin{equation}\label{eq:influence_modified_tribes}
	I_1[f_t] = \bbP (f_t(\omega^1_{+})\neq f_t(\omega^1_{-})) = r_t\, (1-p_t)^{m_t -1} 
	\sim r_t \, \rme^{-1} \,.
\end{equation}
Recalling the definition \eqref{eq:KK-original} of $\cW[f]$,
we then obtain
\begin{equation}\label{eq:W_asymp}
	\mathcal{W}[f_t] \sim t \, m_t \, I_1[f_t]^2
	\sim \rme^{-2}  \, t \, \frac{r_t^2 }{p_t} \,.
\end{equation}

Comparing \eqref{eq:cond1} and \eqref{eq:W_asymp}, we see that to
reach the optimal exponent $\varepsilon/(2-\varepsilon)$ in \eqref{eq:SeV},
the quantities $p_t, r_t, q_{t,\epsilon}$ need to satisfy, as $t\to\infty$,
\begin{equation}\label{eq:optimality_condition_bound}
	\bbcov\bigl[ f(\omega^\epsilon), f(\omega) \bigr] \ge
   \mathcal{W} [f]^{\frac{\varepsilon}{2-\varepsilon} + o(1)}  \qquad \iff \qquad
	q_{t,\epsilon} \geq \bigg( t \, \frac{r_t^2 }{p_t} \bigg)^{\frac{\varepsilon}{2-\varepsilon}+o(1)} \,,
\end{equation}
where the term $o(1)$ in the exponent allows for possible logarithmic corrections.

\begin{remark}[Tribes is not optimal for \eqref{eq:SeV}]
A natural candidate to verify optimality of \eqref{eq:SeV}
is the original Tribes function \cite{BL85} 
corresponding to $A_t = \{\omega_i = 1 \ \forall i=1,\ldots, t\}$.
However, in this case one has $p_t=1/2^t=\rme^{-t\log 2}$, $r_t=1/2^{t-1}\sim 2 \, p_t$ 
and $q_{t,\epsilon}=(1-\frac{\varepsilon}{2})^t = \rme^{t \log(1-\frac{\varepsilon}{2})}$, so 
$q_{t,\epsilon} = p_t^{\gamma_\epsilon} = (t \frac{r_t^2}{p_t})^{\gamma_\epsilon + o(1)}$ with
$\gamma_\epsilon = \frac{-\log(1-\varepsilon/2)}{\log 2}$.
Note that $\gamma_\epsilon \ne \frac{\epsilon}{2-\epsilon}$, and even as $\epsilon \downarrow 0$
we have $\gamma_\epsilon \sim \frac{\epsilon}{2\log 2} \simeq 0.721\,\epsilon$,
hence \eqref{eq:optimality_condition_bound} is not satisfied.
\end{remark}

\subsection{Proof of Lemma~\ref{lem:cond1}}\label{sec:proof_cond1}
We set for short $Y_1 := Y_1(\omega)$ and $Y_1^\epsilon := Y_1(\omega^\epsilon)$.
Observe that
\begin{equation*}
\begin{split}
	\mathbb{E}[(1- f_t(\omega^{\varepsilon})) \, (1-f_t(\omega))]
	&= \bbP(f_t(\omega^\epsilon)=0, \, f_t(\omega) = 0) 
	= \mathbb{P}(Y_1=0,Y_1^{\varepsilon}=0)^{m_t} \\
	& =\mathbb{P}(Y_1=0)^{m_t}\,\mathbb{P}(Y_1^{\varepsilon}=0\, |\, Y_1=0)^{m_t} \\
	&= (1-p_t)^{m_t} \, \big(1 - \mathbb{P}(Y_1^{\varepsilon}=1\,|\, Y_1=0)\big)^{m_t} \,.
\end{split}
\end{equation*}
We can write, recalling $q_{t,\epsilon}$ from \eqref{eq:q_t_def},
\begin{equation*}
\begin{split}
	\mathbb{P}(Y_1^{\varepsilon}=1\,|\, Y_1=0)
	&=\frac{\mathbb{P}(Y_1^{\varepsilon}=1) \big\{1
	-\mathbb{P}(Y_1^{\varepsilon}=1\,|\,Y_1=1)\big\}}{\mathbb{P}(Y_1=0)}
	=\frac{p_t(1-q_{t,\epsilon})}{1-p_t}  \\
	&= p_t(1-q_{t,\epsilon}) + O(p_t^2) \quad \text{as } t \to \infty \,.
\end{split}
\end{equation*}
Finally, since $m_t=\lfloor 1/p_t \rfloor$, we have that $p_tm_t=1+O(p_t)$ and
\begin{align*}
	\bbcov [f_t(\omega^{\varepsilon}),f_t(\omega)]
	& = \bbcov [(1- f_t(\omega^{\varepsilon})) \, (1-f_t(\omega))]\\
	&=(1-p_t)^{m_t}(1-p_t+p_tq_{t,\epsilon}+O(p_t^2))^{m_t} - (1-p_t)^{2m_t}\\
	&=(1-p_t)^{m_t}[\rme^{-1+q_{t,\epsilon}+O(p_t)}-\rme^{-1+O(p_t)}] \\
	&=(1-p_t)^{m_t} \, \rme^{-1} \, [q_{t,\epsilon} + O(p_t) + O(q_{t,\epsilon}^2)] \,.
\end{align*}
Assuming $p_t = o(q_{t,\epsilon})$, we obtain \eqref{eq:cond1}.
\qed

\subsection{Proof of Theorem~\ref{th:sharp}}\label{sec:sharp}
We need to check that the \emph{Modified Tribes} function,
corresponding to $A_t=\big\{ \sum_{i=1}^t \omega_i = a_t\big \}$
with $a_t$ from \eqref{eq:at},
satisfies \eqref{eq:optimality_condition_bound}.

We recall that $p_t = \bbP(A_t)$ is given in \eqref{eq:pt}.
Note that the event $A_t$ means that the difference between the number of ``$+$'' and the number of ``$-$'' signs in $\omega_1,\dots ,\omega_t$ equals $a_t$.

Let us now fix $\epsilon \in (0,1)$ and estimate 
\begin{equation*}
	q_{t,\epsilon}= \mathbb{P}(A_t(\omega^{\varepsilon}) \, | \, A_t(\omega)) \, ,
\end{equation*}
which is the probability that the difference between the number of ``$+$'' and the number of ``$-$'' signs 
is still equal to $a_t$ after resampling.
Assuming that $\sum_{i=1}^t \omega_i = a_t$, set
\begin{equation}\label{eq:N+_N-}
N_t^{+}:=\frac{t+a_t}{2} \quad \text{number of ``$+$'',} \qquad N_t^{-}:=t-N_t^{+}= \frac{t-a_t}{2}\quad \text{number of ``$-$''}.
\end{equation}
Conditionally on $A_t(\omega)$, we can write $A_t(\omega^{\varepsilon})=\{B^+_\epsilon
-B^-_\epsilon=0\}$, 
where $B^+_\epsilon$ and $B^-_\epsilon$ are respectively the number of changed ``$+$'' and ``$-$'' 
signs after rerandomising:
\begin{equation*}
	B^{+}_\epsilon:= |\{i=1,\dots ,t : \omega_i=1, \omega_i^{\varepsilon}=-1\}|, \qquad 
	B^{-}_\epsilon:= |\{i=1,\dots ,t  : \omega_i=-1, \omega_i^{\varepsilon}=1\}|.
\end{equation*}
Under $\mathbb{P}(\cdot\, |\, A_t(\omega))$, $B^{+}_\epsilon$ and $B^{-}_\epsilon$ are independent binomial random variables, with respective laws $\mathrm{Bin}(N_t^{+}, \varepsilon/2)$ and 
$\mathrm{Bin}(N_t^{-}, \varepsilon/2)$ with $N_t^{+}, N_t^{-}$ as in \eqref{eq:N+_N-}. Therefore
\begin{align*}
	\mu_{t,\epsilon} 
	&:= \bbE[B^+_\epsilon - B^-_\epsilon | A_t(\omega)]
	= (N_t^+ - N_t^-) \tfrac{\epsilon}{2} = a_t \, \tfrac{\epsilon}{2} \,,\\
	\sigma_{t,\epsilon}^2 
	&:= \bbvar[B^+_\epsilon - B^-_\epsilon | A_t(\omega)]
	= (N_t^+ + N_t^-) \tfrac{\epsilon}{2} (1-\tfrac{\epsilon}{2})
	= t \, \tfrac{\epsilon}{2} (1-\tfrac{\epsilon}{2}) \,.
\end{align*}
Local Gaussian estimates for the simple random walk then give
\begin{equation}\label{eq:q_t}
	q_{t,\epsilon}=\mathbb{P}(B^+_\epsilon - B^-_\epsilon = 0)
	\sim \mathbb{P}\big(\mathcal{N}(\mu_{t,\epsilon} , \sigma_{t,\epsilon}^2) \in [0,1)\big)
	\sim \frac{\rme^{-\frac{\mu_{t,\epsilon}^2}{2\sigma_{t,\epsilon}^2}}}{\sqrt{2\pi\sigma_{t,\epsilon}^2}}
	=\frac{\rme^{-\frac{a_t^2}{2t}
	\frac{\varepsilon }{2-\varepsilon}}}{\sqrt{2\pi \, 
	t \, \frac{\varepsilon}{2}(1-\frac{\varepsilon}{2})}} \,.
\end{equation}
We observe,
recalling \eqref{eq:at} and \eqref{eq:pt}, that $p_t=o(q_{t,\epsilon})$ as $t\to \infty$
(since $\epsilon/(2-\epsilon) < 1$), so the assumption of Lemma~\ref{lem:cond1}
is satisfied. More precisely,
\begin{equation}\label{eq:q_t_and_p_t}
	q_{t,\epsilon} \sim \frac{(\sqrt{\frac{\pi}{2} \, t} \, p_t)^{\frac{\varepsilon}{2
	-\varepsilon}}}{\sqrt{\frac{\pi}{2} \, t\,\varepsilon(2-\epsilon)}}
	\sim c_\epsilon \, \frac{(t \, p_t)^{\frac{\epsilon}{2-\epsilon}}}{t^{\frac{1}{2-\epsilon}}}
	\qquad \text{with} \quad c_\epsilon \underset{\epsilon\downarrow 0}{\sim} 
	\frac{1}{\sqrt{2\pi \epsilon}} \,.
\end{equation}

Finally, we compute $r_t$ (see \eqref{eq:rt}):
\begin{align}\label{eq:influence_f_tilde}
r_t=\frac{\bbP(\sum_{i=1}^{t-1}\omega_i =a_t + 1)}{2}+ \frac{\bbP(\sum_{i=1}^{t-1}\omega_i =a_t - 1)}{2} \sim \bbP\left(\sum_{i=1}^{t-1}\omega_i =a_t\right) \sim  \, p_t.
\end{align}
The bound in \eqref{eq:optimality_condition_bound} follows by gathering \eqref{eq:q_t_and_p_t} and \eqref{eq:influence_f_tilde}, which shows the optimality of the exponent $\epsilon / (2-\epsilon)$.
More precisely, by \eqref{eq:var_tribes}, \eqref{eq:cond1} and \eqref{eq:W_asymp}, 
\begin{equation}\label{eq:SeV_and_W}
	\frac{\bbcov[f_t(\omega^{\varepsilon}),f_t(\omega)]}{\bbvar [f_t]}
	\sim \tfrac{\rme^{-1}}{1-\rme^{-1}} \, q_{t,\epsilon} \,, \qquad
	\frac{\mathcal{W}[f_t]}{\bbvar [f_t]} \sim \tfrac{\rme^{-1}}{1-\rme^{-1}}  \, t \, p_t \,,
\end{equation}
hence by \eqref{eq:q_t_and_p_t} we get, for a suitable $c'_\epsilon$,
\begin{equation*}
	\frac{\bbcov[f_t(\omega^{\varepsilon}),f_t(\omega)]}{\bbvar [f_t]}
	\sim \frac{c'_\epsilon}{t^{\frac{1}{2-\epsilon}}} \,
	\bigg( \frac{\mathcal{W}[f_t]}{\bbvar [f_t]} \bigg)^{\frac{\epsilon}{2-\epsilon}} \,.
\end{equation*}

At last, from the second relation in \eqref{eq:SeV_and_W} we obtain, recalling \eqref{eq:pt},
\begin{equation*}
	t \sim \bigg( 2 \, \log \frac{\bbvar [f_t]}{\mathcal{W}[f_t]}\bigg)^{\frac{1}{2\gamma}} \,,
\end{equation*}
which plugged into \eqref{eq:SeV_and_W} concludes the proof of \eqref{eq:sharp_bound}. \qed

\appendix

\section{Influences for Boolean functions and binary variables}

We recall the definition of $L^1$ and $L^2$ influences 
for a function $f(\omega) \in L^2$:
\begin{equation} \label{eq:inf12}
	\Inf_k^{(1)}[f] := \bbE \big[ | \delta_k f | \big] \,,
	\qquad \Inf_k^{(2)}[f] := \bbE \big[ (\delta_k f)^2 \big] \,.
\end{equation}
We first show that, for a Boolean function~$f$, these notions coincide up to a factor~$2$.

\begin{lemma}[Influences for Boolean functions] \label{lem:boole}
For a Boolean function $f(\omega) \in \{0,1\}$
\begin{equation}\label{eq:boole12}
	\Inf_k^{(1)}[f] = 2 \, \Inf_k^{(2)}[f] = 
	\bbP\big( f(\omega) \ne f(\omega^k_{\mathrm{ind}}) \big) 
\end{equation}
where $\omega^{k}_{\mathrm{ind}}$ denotes the family $\omega = (\omega_i)_{i\in\bbT}$
with $\omega_k$ replaced by an independent copy~$\omega'_k$.
\end{lemma}

\begin{proof}
For a Bernoulli variable $X$ with mean $q$
we have $\bbE[|X-q|] = 2\, q  (1-q) = 2 \bbvar[X]$.
Conditionally on $(\omega_j)_{j \ne k}$, the distribution of $f(\omega)$ is Bernoulli
with mean $\bbE_k[f]$, hence
\begin{equation*}
	\Inf_k^{(1)}[f] = 2 \, \Inf_k^{(2)}[f] = 2 \,
	\bbE \bigl[ \, \bbE_k[f] \, (1-\bbE_k[f]) \bigr] \,.
\end{equation*}
Using the modified family $\omega^k_{\mathrm{ind}}$, we can write
$\bbE_k[f] \, (1-\bbE_k[f]) = \bbE_k[f(\omega) \, (1-f(\omega^k_{\mathrm{ind}}))]$,
and recalling that $f \in \{0,1\}$ is Boolean we obtain
(with $\bbP_k(\cdot) := \bbE_k[\ind_{\{\cdot\}}]$)
\begin{equation*}
	\bbE_k[f] \, (1-\bbE_k[f]) = 
	\bbP_k( f(\omega) = 1, f(\omega^k_{\mathrm{ind}}) = 0)
	= \frac{1}{2} \, \bbP_k\big( f(\omega) \ne f(\omega^k_{\mathrm{ind}}) \big) \,.
\end{equation*}
This completes the proof of \eqref{eq:boole12}.
\end{proof}

\smallskip

We next compute the influences in the special case of binary $\omega_i$'s.

\begin{lemma}[Influences for binary variables] \label{lem:binary}
If $\omega_i \in \{x_-, x_+\}$ are binary variables, see \eqref{eq:binary0},
for any function $ f(\omega)$ we have
\begin{align}
	\label{eq:inf-binary1}
	\Inf_k^{(1)} [f] = 2 \, p(1-p) \, \bbE\big[|f(\omega^k_+) - f(\omega^k_-)| \big] \,, \\
	\label{eq:inf-binary2}
	\Inf_k^{(2)} [f] =  p(1-p) \, \bbE\big[ \big( f(\omega^k_+) - f(\omega^k_-) \big)^2 \big] \,,
\end{align}
where $\omega^k_\pm$ denotes the family $\omega = (\omega_i)$ in which we
fix $\omega_k = x_\pm$.
\end{lemma}

\begin{proof}
By \eqref{eq:binary0} we compute $\bbE_k[f] = p \, f(\omega^k_+) + (1-p) \, f(\omega^k_-)$, hence
\begin{equation}\label{eq:delta-binary}
	\delta_k f = f - \bbE_k[f] = \begin{cases}
	(1-p) \, \big\{ f(\omega^k_+) - f(\omega^k_-) \big\}
	& \text{if } \omega_k = x_+ \,, \\
	-p \, \big\{ f(\omega^k_+) - f(\omega^k_-) \big\}
	& \text{if } \omega_k = x_- \,,
	\end{cases}
\end{equation}
from which \eqref{eq:inf-binary1} and \eqref{eq:inf-binary2} readily follow.
\end{proof}

\begin{remark}\label{rem:Tal}
Relation \eqref{eq:delta-binary}
shows that, for binary $\omega_i$'s, our definition of
$\delta_k f = f - \bbE_k[f]$ coincides with $\Delta_k f$ from \cite{Tal}.
\end{remark}

Combining Lemmas~\ref{lem:boole} and~\ref{lem:binary},
we finally obtain the following result.

\begin{lemma}[Influences for Boolean functions of binary variables] \label{lem:binary1}
For a Boolean function $f(\omega) \in \{0,1\}$ of binary variables $\omega_i \in \{x_-, x_+\}$,
see \eqref{eq:binary0}, we have
\begin{equation} \label{eq:inf-boolean}
	\Inf_k^{(1)}[f] = 2 \, \Inf_k^{(2)}[f] 
	= 2 \, p(1-p) \, \bbP \big( f(\omega^k_+) \ne f(\omega^k_-) \big) \,.
\end{equation}
\end{lemma}

\begin{proof}
We apply \eqref{eq:boole12} and note that
\begin{equation*}
	\{f(\omega) \ne f(\omega^k_{\mathrm{ind}})\}
	= \{	f(\omega^k_+) \ne f(\omega^k_-) \} \cap
	\{\omega_k \ne \omega'_k\} \,,
\end{equation*}
where we recall that $\omega'_k$ denotes an independent copy of $\omega_k$.
Since $\bbP(\omega_k \ne \omega'_k) = 2p(1-p)$ and
the event $\{ f(\omega^k_+) \ne f(\omega^k_-) \}$ is independent of
$\omega_k, \omega'_k$, the proof is completed.
\end{proof}

\section{Orthogonal decomposition}
\label{app:ES}

We first prove the Efron-Stein decomposition in
Proposition~\ref{th:genchaos}, which
shows that any function $f(\omega) \in L^2$ admits the chaos
expansion \eqref{eq:f-chaos} with $f_I$ given by \eqref{eq:fIdef}.

\begin{proof}[Proof of Proposition~\ref{th:genchaos}]
The necessity of \eqref{eq:fIdef} is easy: if we
assume that \eqref{eq:f-chaos} holds for some functions $f_I$'s satisfying
\eqref{eq:fI}, we already observed that the properties in \eqref{eq:deltakEJf} must hold,
from which we can directly deduce \eqref{eq:fIdef}.

It remains to show that \eqref{eq:fIdef} holds with $f_I$ given by \eqref{eq:fIdef}.
We first assume that $\bbT$ is finite, say $\bbT = \{1,\ldots, n\}$.
Since $\bbE_i + \delta_i$ is the identity operator, see \eqref{eq:deltak}, we can write
\begin{equation*}
	f = (\bbE_{1} + \delta_{1}) \cdots (\bbE_{n} + \delta_{n}) \, f
	= \bbE[f] + \sum_{d=1}^n \sum_{I \subseteq \bbT:\, |I| = d} f_I \, ,
\end{equation*}
where we expanded the product of \emph{commuting} operators $\delta_i$ and $\bbE_j$
(by Fubini's theorem) and we denoted by
$f_I$ the function obtained from~$f$ applying $\delta_i$ for $i \in I$ and $\bbE_j$ for $j \in I^c$.
If we write $I = \{i_1, \ldots, i_d\}$ and, for convenience, $I^c = \{j_1, \ldots, j_{n-d}\}$,
we obtain \eqref{eq:fIdef}:
\begin{equation*}
	f_I \,=\, \delta_{i_1} \cdots\, \delta_{i_d} \,\,  \tilde{f}_I \qquad \text{with} \qquad
	\tilde{f}_I \,:=\, \bbE_{j_1} \cdots \, \bbE_{j_{n-d}} \, f
	\,=\, \bbE[ \, f  \, | \, \cF_I] \,.
\end{equation*}

We next consider the case $|\bbT| = \infty$.
We write $\bbT = \bigcup_{n\in\N} \bbT_n$ for increasing sets $\bbT_n \subseteq \bbT_{n+1}$
with $|\bbT_n| = n$
and we define
\begin{equation} \label{eq:fn}
	f_n := \bbE[ \,f\, |\cF_n] \qquad \text{where} 
	\quad \cF_n := \sigma(\omega_i \colon i \in \bbT_n) \,.
\end{equation}
Since $|\bbT_n|<\infty$, we already know that $f_n$ admits the decomposition \eqref{eq:f-chaos}.
Crucially $(f_n)_I = f_I$ for $I \subseteq \bbT_n$ because
$\bbE[ f_n \,|\, \cF_{I}] =
\bbE[\bbE[f|\cF_n] \,|\, \cF_{I}] = \bbE[f\,|\,\cF_I]$, see \eqref{eq:fIdef}, hence
\begin{equation}\label{eq:f-chaos-n}
	f_n =\, \bbE[f] + \sum_{d = 1}^\infty \,
	\sum_{I \subseteq \bbT_n :\, |I| = d} f_I \, .
\end{equation}
Note that $f_n = \bbE[\,f\,|\cF_n]$ is a martingale bounded in $L^2$, 
hence $f_n \to \bbE[\,f\,|\cF_\infty] = f$ in $L^2$
because $\cF_\infty = \sigma(\bigcup_{n\in\N}\cF_n) = \sigma(\omega = (\omega_i)_{i\in\bbT})$.
Letting $n\to\infty$ in \eqref{eq:f-chaos-n}, we obtain \eqref{eq:f-chaos}.
\end{proof}

We deduce that any function $f(\omega) \in L^2$ of binary variables $\omega_i$'s
is a polynomial chaos.

\begin{lemma}[Completeness for binary $\omega_i$'s]\label{lem:completeness}
Let the $\omega_i$ be independent and centred real random variables which take two values.
Then every function $f(\omega) \in L^2$ is a polynomial chaos,
i.e.\ it is of the form \eqref{eq:fpoly}.
\end{lemma}

\begin{proof}
When $|\bbT| = n < \infty$, functions $f(\omega)$ with binary $\omega_i$'s
can be identified with functions $f: \{x_-, x_+\}^n \to \R$,
which form a vector space of dimension $2^n$.
The monomials $(\omega_I)_{I\subseteq \bbT}$ are precisely $2^n$ (orthogonal, hence)
linearly independent functions, which are then a basis.
It follows that any function $f(\omega)$ is a (finite) sum
$\sum_{I} \hat{f}(I) \, \omega_I$, i.e.\ a polynomial chaos.

When $|\bbT| = \infty$ we argue as in the proof of Proposition~\ref{th:genchaos}:
we write $\bbT = \bigcup_{n\in\N} \bbT_n$
for finite increasing sets $\bbT_n$ and define $f_n$ by \eqref{eq:fn}.
We already know that $f_n(\omega) = \sum_{I \subseteq \bbT_n} \hat{f_n}(I) \, \omega_I$ 
is a polynomial chaos with coefficients $\hat{f_n}(I) = \hat{f}(I)$ 
independent of~$n$, 
since $\langle f_n, \omega_I\rangle = \langle f, \omega_I\rangle$
for $I \subseteq \bbT_n$.
Then \eqref{eq:fpoly} holds because $f_n \to f$ in $L^2$ by martingale convergence.
\end{proof}

\section{Assumption~\ref{hyp:gen2}, ensembles, and hypercontractivity}
\label{app:hyper}

In this section, we discuss Assumption~\ref{hyp:gen2}, connecting it
to the notion of \emph{ensembles} from \cite{MOO}.
We then prove Theorem~\ref{th:hyper}, deducing it from results in \cite{MOO}.

\subsection{Rephrasing Assymption~\ref{hyp:gen2}}

Let us look more closely at our key Assumption~\ref{hyp:gen2}. We start
from part~\eqref{it:a}, discussing property \eqref{eq:belongs}
relative to the vector spaces $\cV_i \subseteq L^2(E_i, \mu_i)$,
which are by assumption separable Hilbert spaces with $1 \in \cV_i$.

Let us fix an \emph{orthonormal basis} of $\cV_i$ starting with the constant function~$1$:
\begin{equation} \label{eq:basis}
	h^{(0)}_i := 1\,, \ h^{(1)}_i\,, \ h^{(2)}_i\,, \ \ldots
	\ \in \cV_i \subseteq L^2(E_i,\mu_i) \colon
	\qquad
	\langle h^{(\ell)}_i , \, h^{(\ell')}_{i} \rangle_{E_i} = \delta_{\ell,\ell'} \,.
\end{equation}
If $\cV_i$ has finite dimension $M_i < \infty$
the sequence is finite: $h^{(\ell)}_i$ for $\ell = 0, 1, \ldots, M_i-1$
(otherwise it is an infinite sequence). 
We can then rephrase condition \eqref{eq:belongs} as follows:
for any given index $i\in\bbT$ we can write, for a.e.\ $(\omega_j)_{j\ne i}$,
\begin{equation}\label{eq:belongs2}
	f(\omega) = \! \sum_{\ell = 0, 1, \ldots} \!\! \alpha_\ell \, h_i^{(\ell)}(\omega_i) \qquad
	\text{for coefficients} \ \, \alpha_\ell = \alpha_\ell\bigl( (\omega_j)_{j\ne i} \bigr) \,.
\end{equation}
This means that $f(\omega)$, as a function of~$\omega_i$, 
is a linear combination of the functions $h_i^{(\ell)}$'s.

\smallskip

We next discuss Assumption~\eqref{hyp:gen2}~\eqref{it:b}, focusing on the
bound \eqref{eq:Mq}. The next simple result covers
many cases, including Examples~\ref{ex:fin-supp} and~\ref{ex:poly-chaos} (and beyond). 

\begin{lemma}[Finite-dimensional vector space]\label{th:finite-gen}
Let the random variables $(\omega_i)_{i\in\bbT}$ be i.i.d.\
with law $\mu$ on $(E,\cE)$.
Fix $q > 2$ and a \emph{finite dimensional} vector space
$\cV \subseteq L^q(E,\mu)$.

Any function $f(\omega) \in L^2$ which satisfies 
Assumption~\ref{hyp:gen2}~\eqref{it:a} with $\cV_i = \cV$
also satisfies Assumption~\ref{hyp:gen2}~\eqref{it:b}, for a suitable $M_q  < \infty$.
\end{lemma}

\begin{proof}
We need to show that condition \eqref{eq:Mq} holds for some $M_q < \infty$.
Fix an orthonormal basis $h^{(0)}, h^{(1)}, \ldots, h^{(k)}$ of $\cV$
with $h^{(0)} = 1$. Any $g \in \cV$ with $\bbE[g(\omega_i)] = \langle g, h^{(0)} \rangle = 0$
can be written as $g(\cdot) = \sum_{\ell=1}^k \alpha_\ell \, h^{(\ell)}(\cdot)$
with $\alpha_\ell \in \R$. By the triangle inequality and Cauchy-Schwarz
\begin{equation*}
	\| g(\omega_i) \|_q  \le \sum_{\ell=1}^k |\alpha_\ell| \, \|h^{(\ell)}(\omega_i)\|_q
	\le M_q \, \bigg(\sum_{\ell=1}^k \alpha_\ell^2 \bigg)^{\frac{1}{2}} \qquad \text{with} \ \
	M_q := \bigg(\sum_{\ell=1}^k \|h^{(\ell)}(\omega_i)\|_q^2 \bigg)^{\frac{1}{2}} < \infty \,.
\end{equation*}
Since $\|g(\omega_i)\|_2^2 = \sum_{\ell=1}^k \alpha_\ell^2$, we have shown that \eqref{eq:Mq} holds.
\end{proof}

\begin{remark}\label{rem:appl-fin}
Lemma~\ref{th:finite-gen} covers Example~\ref{ex:fin-supp}, 
because for $\mu$ with finite support
the space $\cV = L^2(E,\mu)$
is finite dimensional (and it coincides with $L^q(E,\mu)$ for any~$q$).

It also includes Example~\ref{ex:poly-chaos}, 
since the space of linear functions
$\cV = \{\alpha + \beta x \colon \alpha, \beta \in \R\}$ has dimension~$2$ and
$\cV \subseteq L^q(\R, \mu)$
under assumption \eqref{eq:omega}.
\end{remark}

\begin{remark}\label{rem:appl-fin2}
We can invoke Lemma~\ref{th:finite-gen} also for Remark~\ref{rem:high},
namely when $f(\omega) \in L^2$ depends on any given $\omega_i$ as
a polynomial of \emph{uniformly bounded degree}, say at most~$h$
(for $h=1$ we recover polynomial chaos from Example~\ref{ex:poly-chaos}).

Indeed,
Assumption~\ref{hyp:gen2}~\eqref{it:a} is satisfied with
$\cV_i = \cV = \{\alpha_0 + \alpha_1 x + \ldots + \alpha_h x^h \colon \alpha_i \in \R\}$ which
has finite dimension $h+1$. In order to have $\cV \subseteq L^q(\R,\mu)$ for some
given~$q > 2$, we only need to require that the $\omega_i$'s have uniformly
bounded moments of order~$qh$.
\end{remark}

\subsection{Ensembles and multi-linear polynomials}
\label{sec:ensembles}

Recall the orthonormal basis \eqref{eq:basis} in the space $\cV_i$.
Defining the families of random variables
\begin{equation}\label{eq:Xi}
	\mathscr{X}_i := \big\{
	X_{i,0} := 1 \,,\  X_{i,1} \,, \ X_{i,2}  \,, \ \ldots \,\big\} \subseteq
	L^2(\Omega) \colon \qquad
	X_{i,\ell} := h_i^{(\ell)}(\omega_i) \,,
\end{equation}
we obtain a so-called \emph{sequence of ensembles}
$(\mathscr{X}_i)_{i\in\bbT}$ according to \cite[Definition~3.1]{MOO}:
\begin{itemize}
\item random variables $X_{i,\ell}$ within each family $\mathscr{X}_i$ are \emph{orthonormal};
\item different families $\mathscr{X}_i$'s are \emph{independent}.
\end{itemize}

Let us call \emph{multi-index} any sequence
$\bell = (\ell_i)_{i\in\bbT}$ with $\ell_i \in \{0,1,\ldots\}$ such that
\emph{$\ell_i \ne 0$ only for finitely many $i\in\bbT$}
(if $\cV_i$ has finite dimension $M_i < \infty$ we also require $\ell_i < M_i$).
For each multi-index $\bell$ we define a \emph{multi-linear monomial $X_{\bell}$ in the
ensembles $(\mathscr{X}_i)_{i\in\bbT}$}:
\begin{equation}\label{eq:Xell}
	X_{\bell} := \prod_{i\in\bbT} X_{i,\ell_i}
	= \prod_{i\in\bbT} h_i^{(\ell_i)}(\omega_i) \,,
\end{equation}
(the product is finite since $X_{i,\ell_i} = 1$ when $\ell_i = 0$).
These random variables
are orthonormal:
\begin{equation*}
	\langle X_{\bell}, X_{\bell'} \rangle = \delta_{\bell,\bell'}
	\qquad \text{for all multi-indexes $\bell,\bell'$.}
\end{equation*}

\begin{definition}[\cite{MOO}]
We call \emph{multi-linear polynomial in the ensembles
$(\mathscr{X}_i)_{i\in\bbT}$} any random variable $f(\omega)$ in the linear
subspace of $L^2(\Omega)$ generated by the $X_{\bell}$'s:
\begin{equation}\label{eq:chaos-gen}
	f(\omega) = \sum_{\bell} \hat{f}(\bell) \, X_{\bell} \qquad
	\text{for real coefficients $\big( \hat{f}(\bell) \big)_{\bell}$ with }
	\sum_{\bell} \hat{f}(\bell)^2 < \infty 
\end{equation}
where the series converges in $L^2$ (hence we must have
$\hat{f}(\bell) = \langle f(\omega), X_{\bell}\rangle$).
\end{definition}

We are ready to give an equivalent characterisation of Assumption~\ref{hyp:gen2}~\eqref{it:a}.

\begin{lemma}\label{th:tensor}
A function $f(\omega) \in L^2$ satisfies Assumption~\ref{hyp:gen2}~\eqref{it:a}
if and only if it is a multi-linear
polynomial in the ensembles $(\mathscr{X}_i)_{i\in\bbT}$, i.e.\
\eqref{eq:chaos-gen} holds.
In this case, the functions $f_I(\omega)$ appearing in 
the chaos decomposition \eqref{eq:f-chaos} of~$f(\omega)$ are
\begin{equation} \label{eq:fIalt}
	f_I(\omega) = \sum_{\bell \colon
	\substack{\ell_i \ge 1 \, \forall i \in I \\ \ell_j = 0 \, \forall j \not\in I}} \hat{f}(\bell) \, X_{\bell}
	\qquad \text{with} \quad \hat{f}(\bell) = \langle f(\omega), X_{\bell} \rangle \,.
\end{equation}
\end{lemma}

 Before giving the proof, let us make two observations.

\begin{remark}[Ensembles and polynomial chaos]
For real valued $\omega_i$'s
with $\bbE[\omega_i] = 0$ and $\bbE[\omega_i^2] = 1$, if we take
the space of linear functions $\cV_i = \{ x \mapsto \alpha + \beta x \colon
\alpha,\beta \in \R\}$
with canonical basis
$h_i^{(0)} = 1$ and $h_i^{(1)}(x) = x$, we obtain the ensembles
$\mathscr{X}_i = \{1, \omega_i\}$. In this setting, 
multi-indexes $\bell$ have components $\ell_i \in \{0,1\}$ and
$X_{\bell} = \omega_{i_1} \cdots \omega_{i_d}$ is 
nothing but the multi-linear monomial in \eqref{eq:fpoly}
with $\{i_1, \ldots, i_d\} = \{i\in\bbT \colon \ \ell_i = 1\}$.
Definition \eqref{eq:chaos-gen} then gives the polynomial
chaos in \eqref{eq:fpoly}.
\end{remark}

\begin{remark}[Ensembles and chaos decomposition]\label{rem:ViL2}
If we take $\cV_i = L^2(E_i,\mu_i)$ --- assumed to be separable ---
Assumption~\ref{hyp:gen2}~\eqref{it:a} is satisfied by \emph{every $f(\omega) \in L^2$}
(by Fubini's theorem). Then it follows by Lemma~\ref{th:tensor}
that \emph{every $f(\omega) \in L^2$ is a multi-linear
polynomial}, 
i.e.\ multi-linear monomials $X_{\bell}$
are a basis of $L^2(\Omega,\sigma((\omega_i)_{i \in \bbT}), \bbP)$.
Formula \eqref{eq:fIalt} then provides a hands-on construction of the chaos decomposition
in Proposition~\ref{th:genchaos}.
\end{remark}

\begin{proof}[Proof of Lemma~\ref{th:tensor}]
It is clear that any multi-linear polynomial \eqref{eq:chaos-gen} 
in the ensembles $(\mathscr{X}_i)_{i\in\bbT}$ satisfies 
Assumption~\ref{hyp:gen2}~\eqref{it:a}, since by construction $X_{i,\ell} = h_i^{(\ell)}(\omega_i)$
with $h_i^{(\ell)} \in \cV_i$.

We now prove that any $f(\omega) \in L^2$ satisfying Assumption~\ref{hyp:gen2}~\eqref{it:a}
is of the form \eqref{eq:chaos-gen}.
We fix an index $i\in\bbT$ and note
that, by \eqref{eq:belongs}, $\omega_i \mapsto f(\omega)$ belongs to $\cV_i$
conditionally on $(\omega_j)_{j\ne i}$.
Since $h^{(0)}_i, h^{(1)}_i, \ldots$ is an orthonormal basis
of $\cV_i$, we write $f(\omega)$ as the $L^2$ convergent series
\begin{equation} \label{eq:fseries}
	f(\omega) = \sum_{\ell_i=0,1,\ldots} 
	\langle f(\omega) , \, h^{(\ell_i)}_i(\omega_i) \rangle_{E_i} \, 
	h_i^{(\ell_i)}(\omega_i) =
	\sum_{\ell_i=0,1,\ldots} \bbE_i\big[ f(\omega) \, X_{i,\ell_i} \big] \, 
	X_{i,\ell_i} \,. 
\end{equation}
If $i' \in \bbT$ is another index, using again Assumption~\ref{hyp:gen2}~\eqref{it:a}
and recalling \eqref{eq:Xell}, we obtain
\begin{equation*}
\begin{split}
	f(\omega) 
	&= \sum_{\ell_{i'}=0,1,\ldots} \, \sum_{\ell_i =0,1,\ldots} \bbE_{i'} \big[
	\bbE_i\big[ f(\omega) \, X_{i,\ell_i} \big] \, X_{i',\ell_{i'}} \big] \, 
	X_{i,\ell_i} \, X_{i',\ell_{i'}} \\
	&= \sum_{\substack{\bell\ \text{multi-indexes} \\
	\ell_k = 0 \, \forall k \not\in \{i,i'\}}}
	\bbE_{i'} \, \bbE_{i} \big[ f(\omega) \, X_{\bell} \big] \, X_{\bell} \,.
\end{split}
\end{equation*}
Iterating the same argument for all indexes
in a given finite set $\bbT_n = \{i_1, \ldots, i_n\} \subseteq \bbT$,
since $\bbE_{i_1} \bbE_{i_2} \cdots\, \bbE_{i_n} [\,\cdot\,] 
= \bbE[\,\cdot\,|\,\cF_{\bbT \setminus \bbT_n}]$
by Fubini's theorem, we obtain
\begin{equation} \label{eq:fquasin}
	f(\omega) = \sum_{\substack{\bell \ \text{multi-indexes} \\
	\text{supported in } \bbT_n}}
	\bbE\big[ f(\omega) \, X_{\bell} \,\big|\, \cF_{\bbT \setminus \bbT_n} \big] \, X_{\bell} \,.
\end{equation}
If $\bbT$ is finite, we can take $\bbT_n = \bbT$ and we obtain our goal \eqref{eq:chaos-gen},
since $\bbE[ \,\cdot \,|\, \cF_{\bbT \setminus \bbT_n} ] = \bbE[\cdot]$.

When $\bbT$ is infinite, 
writing $\bbT = \bigcup_{n\in\N} \bbT_n$
for finite increasing sets $\bbT_n$, we have by \eqref{eq:fquasin}
\begin{equation} \label{eq:fquasin2}
	\bbE [ f(\omega) \,|\, \cF_{\bbT_n}] = \sum_{\substack{\bell \ \text{multi-indexes} \\
	\text{supported in } \bbT_n}}
	\bbE\big[ f(\omega) \, X_{\bell} \big] \, X_{\bell} 
\end{equation}
because $\bbE [ \, \bbE[ \,\cdot \,|\, \cF_{\bbT \setminus \bbT_n} ] \,|\, \cF_{\bbT_n}]
= \bbE[\cdot]$ by Fubini's theorem. Since
$\bbE[ f(\omega) \, |\, \cF_{\bbT_n}] \to f(\omega)$
in $L^2$ as $n\to\infty$, by martingale convergence,
our goal \eqref{eq:chaos-gen} follows by \eqref{eq:fquasin2}.
\end{proof}

\subsection{Proof of Theorem~\ref{th:hyper}}

Let us prove \eqref{eq:hyper0}. 
By Lemma~\ref{th:tensor}, any $f(\omega) \in L^2$ satisfying Assumption-\ref{hyp:gen2}
is a multi-linear polynomial in the sequence of ensembles $(\mathscr{X}_i)_{i\in \bbT}$.
As a consequence, in the language of \cite[Definition~3.9]{MOO},
proving the bound \eqref{eq:hyper0} amounts to showing that
\emph{the sequence of ensembles $(\mathscr{X}_i)_{i\in \bbT}$ is $(2,q,\eta_q)$-hypercontractive}.

We first consider the case when $\bbT$ is finite
and all $\cV_i$'s have finite dimension, so each ensemble $\mathscr{X}_i$
consists of finitely many random variables. This is the setting considered in \cite[Proposition~3.11]{MOO},
by which it suffices to show that 
\emph{each individual family $\mathscr{X}_i$ is $(2,q,\eta_q)$-hypercontractive},
i.e.\ the bound \eqref{eq:hyper-gen} holds
(the proof is similar to \cite[Lemma~5.3]{Jan}).
But this is precisely
the way we have fixed $\eta_q \in (0,1)$, see Lemma~\ref{th:hyp},
hence \eqref{eq:hyper0} is proved in this case.

We next consider the case when either $\bbT$ is infinite or
some $\cV_i$ has infinite dimension. Let us write $\bbT = \bigcup_{n\in\N} \bbT_n$
as the increasing union of finite sets $\bbT_n$.
For each $n\in\N$, we define $f_n(\omega) \in L^2$
by restricting the sum in \eqref{eq:chaos-gen} to multi-indexes
$\bell = (\ell_i)_{i\in\bbT}$ supported in $\bbT_n$ and
with all components satisfying $\ell_i \le n$.
We already proved that \eqref{eq:hyper0} holds for $f_n$:
\begin{equation}\label{eq:quasilim}
	\| T^{\eta_q} f_n \|_q \le  \| f_n \|_2 
	\qquad \forall n \in\N \,,
\end{equation}
and it remains to let $n\to\infty$.
By construction $f_n \to f$ in $L^2$, and since $T^\eta$ is a bounded operator,
we also have $T^{\eta_q} f_n \to T^{\eta_q} f$ in $L^2$,
in particular $|T^{\eta_q} f_n|^q \to |T^{\eta_q} f|^q$ in probability.
Taking the limit $n\to\infty$ in \eqref{eq:quasilim} we then obtain, by Fatou,
\begin{equation*}
	\| T^{\eta_q} f \|_q \le \liminf_{n\to\infty}
	\| T^{\eta_q} f_n \|_q \le \lim_{n\to\infty}  \| f_n \|_2 = \|f\|_2 \,,
\end{equation*}
which completes the proof of \eqref{eq:hyper0}.

\smallskip

The proof just given shows that the bound \eqref{eq:hyper0} holds not just for~$f$
but for any multi-linear polynomial in the sequence of ensembles $(\mathscr{X}_i)_{i\in \bbT}$,
which includes any linear combination of components $f_I$ such as $f'$ in \eqref{eq:f'}.

If we now consider $f'_d = \sum_{|I| \le d} \alpha_I \, f_I$ with degree at most~$d$,
\emph{we can define $T^{1/\eta_q} \, f'_d$ even though $1/\eta_q > 1$},
see \eqref{eq:Teta}, and we have the identity $f'_d = T^{\eta_q} \, \big(T^{1/\eta_q} \, f'_d\big)$.
Applying \eqref{eq:hyper0} with $f$ replaced by $T^{1/\eta_q} \, f'_d$,
which is still a linear combination of $f_I$'s, we get
\begin{equation*}
	\|f'_d\|_q^2 \le \|T^{1/\eta_q} \, f'_d\|_2^2
	= \sum_{I \subseteq \bbT \colon |I| \le d} \, \frac{1}{\eta_q^{2|I|}} \, \alpha_I^2 \, \| f_I \|_2^2
	\le  \frac{1}{\eta_q^{2d}} \, \sum_{I \subseteq \bbT \colon
	|I| \le d} \alpha_I^2 \, \| f_I \|_2^2
	= \frac{1}{\eta_q^{2d}} \, \|f'\|_2^2 \,,
\end{equation*}
which is precisely \eqref{eq:hyper}.
\qed

\end{document}